\documentclass[3p]{elsarticle}

\usepackage{lineno,hyperref}
\usepackage{color}

\modulolinenumbers[5]

\journal{arXiv}

\usepackage{amsfonts}
\usepackage{amsmath, amscd,amssymb,bm,amsbsy,epsf}
    \usepackage{enumerate}
    \usepackage{paralist}
\usepackage{bbm, dsfont}
\usepackage{graphicx}
\usepackage{subfigure}
\usepackage{mathrsfs}
\usepackage{verbatim}
\usepackage{inputenc}
\usepackage{bbm}
\usepackage{multicol}
\usepackage{balance}

\newtheorem{theorem}{Theorem}
\newtheorem{lemma}[theorem]{Lemma}

\newtheorem{proposition}[theorem]{Proposition}
\newdefinition{definition}{Definition}
\newdefinition{hypothesis}{Hypothesis}
\newdefinition{remark}{Remark}
\newproof{proof}{Proof}
\DeclareMathOperator*{\interior}{int}
\DeclareMathOperator*{\cl}{cl}

\def\div{\bm{\nabla} \cdot}
\def\grad{\bm{\nabla}}

\def\defining{\overset{\mathbf{def} } = }

\def\outer{\bm {\widehat{\nu} } }
\def\ind{\boldsymbol{\mathbbm{1}}}
\def\n{\bm{\widehat{ n} } }

\def\N{\bm{\mathbbm{N} } }
\def\R{\bm{\mathbbm{R}}}

\def\A{{\mathcal A}}
\def\B{{\mathcal B}}
\def\C{{\mathcal C}}
\def\E{{\mathcal E}}

\def\V{\mathbf V}

\def\H{\mathbf H}
\def\X{\mathbf X}
\def\Y{\mathbf Y}
\def\L{\mathbf L}
\def\Hdiv{\mathbf{H_{div}}}

\def\g{\mathbf g}

\def\u{\mathbf u}

\def\v{\mathbf v}
\def\w{\mathbf w}
\def\x{\mathbf x}
\def\y{\mathbf y}

\def\stress{f_{\Sigma } } 
\def\flux{f_{\bm{\hat{n}} } }

\def\uone{\mathbf{u}_{1} }

\def\utwo{\mathbf{u}_{2} }

\def\wone{\mathbf{w}_{1} }
\def\wtwo{\mathbf{w}_{2} }
\def\vone{\mathbf{v}_{1} }
\def\vtwo{\mathbf{v}_{2 } }

\def\pone{p_{ 1 } }

\def\ptwo{p_{ 2} }

\def\qone{q_{ 1 } }
\def\qtwo{q_{ 2 } }
\def\rone{r_{ 1 } }
\def\rtwo{r_{ 2 } }

\def\Omeone{\Omega_{1}}
\def\Ometwo{\Omega_{2}}

\def\O{\mathcal{O}}
\def\U{\mathcal{U}}
\def\M{\mathcal{M}}
\def\triang{\mathcal{G}}
\def\tone{\mathcal{G}_{ 1 } }
\def\ttwo{\mathcal{G}_{ 2 } }

\def\edges{\mathcal{E}}

\begin{document}

\begin{frontmatter}

\title{The Interaction Between PDE and Graphs \\
in Multiscale Modeling}
\tnotetext[mytitlenote]{This material is based upon work supported by project HERMES 27798 from Universidad Nacional de Colombia,
Sede Medell\'in.}

%% Group authors per affiliation:
\author{Fernando A. Morales \sep $\&$ \sep Sebasti\'an Naranjo \'Alvarez. %\fnref{myfootnote}
}
\address{Escuela de Matem\'aticas
Universidad Nacional de Colombia, Sede Medell\'in \\
Calle 59 A No 63-20 - Bloque 43, of 106,
Medell\'in - Colombia}
%\fntext[myfootnote]{Since 1880.}

%% or include affiliations in footnotes:
\author[mymainaddress]{Fernando A Morales}
%\ead[url]{www.unal.edu.co}

%\author[mysecondaryaddress]{Global Customer Service\corref{mycorrespondingauthor}}
\cortext[mycorrespondingauthor]{Corresponding Author}
\ead{famoralesj@unal.edu.co}

%\address[mymainaddress]{Universidad Nacional de Colombia, Sede Medell\'in}
%\address[mysecondaryaddress]{360 Park Avenue South, New York}

\begin{abstract}
%We will set a PDE coupled system in a bipartite graph with boundary conditions as a tool for multiscale modeling between continuous and discrete data.
In this article an upscaled model is presented, for complex networks with highly clustered regions exchanging some abstract quantities in both, microscale and macroscale level. Such an intricate system is approximated by a partitioned open map in $\R^{2}$ or $\R^{3}$. The behavior of the quantities is modeled as flowing in the map constructed and thus it is subject to be described by partial differential equations. We follow this approach using the Darcy Porous Media, saturated fluid flow model in mixed variational formulation.
\end{abstract}


\end{frontmatter}

%\linenumbers

%%%%%%%%%%%%%%%%%%%%%%%%%%%%%%%%%%%%%%%%%%%%%%%%%%%%%%%%
%%%%%%%%%%%%%%%%%%%%%%%%%%%%%%%%%%%%%%%%%%%%%%%%%%%%%%%%
\section{Introduction}
A highly clustered network, is a graph such that its clustering coefficient is close to one. 
The clustering coefficient \cite{EstradaComplexNetworks} of a vertex $v$ in a graph is defined as the number of triangles connected to $v$ divided by the number of triples where $v$ is incident on two edges (triples ``centered" at $v$). The network clustering coefficient is the average value of the clustering for every node in the graph. In this work we consider a simple, connected network where two nodes are linked if and only if they exchange some abstract quantities and containing highly clustered regions or universes (communities, depending on the context). Hence, these universes also exchange the quantities of interest with other clustered regions in the network or multiverse, but at a different scale. Therefore, a macroscale graph representing each clustered region with a node and connecting two nodes whenever the universes they represent trade these quantities, is a natural upscaled model. Additionally, the number of nodes in each universe is very large compared to the number of clustered regions in the multiverse and the exchange of abstract quantities between universes is very large compared to the exchange amongst the individual nodes in the universe. A natural example is the exchange of goods and services amongst members of a nation and between the nations; other real world systems resembling these characteristics can be found in biotechnology \cite{ClusteringBiotechnology}, social and economic networks \cite{Jackson}, the Internet \cite{EstradaComplexNetworks, ComplexNetworksMarteenVanSteen}, etc. The aim of this article is to provide an overall description of how these quantities are exchanged at the macroscale level. Such necessity has already been stated implicitly in \cite{ComplexNetworksMarteenVanSteen} according to the quote: ``... what makes these networks complex is that they are generally so huge that it is impossible to understand or predict their overall behavior by looking into the behavior of individual nodes or links...". On the other hand, a highly descriptive model becomes impractical for complex networks because of the elevated computational costs, numerical instability and low quality solutions introduced by large-scale computations.  %\\

The relationship between PDE and graphs has been subject of study in recent years. Most of the work has been done to provide the basic definitions of the domain associated with the graph and the strong differential operators defined by the PDE of interest, see \cite{GraphsPDE} for a global survey on the field. %This is particularly difficult since graphs are abstract entities by definition thus, extending the methods of Calculus proves imposible unless we adapt the graph to a given application. 
Authors commonly choose the 1-D simplicial complex given by an embedding in $\R^{N}$ of the studied graph, define strong operators in the edges and matching conditions on the vertices, together with appropriate function spaces. The mathematical approach is mainly classical and the technique heavily relies on the eigenfunction-eigenvalue expansion methods and/or maximum-minimum principles; as an example of this treatment see \cite{WaveGraphs}, for a broad exposition see \cite{KutchmentBerkolaiko}. The results depend crucially on the geometry of the embedding however, it is not clear how to make such choice. Seeking to gain independence from this limitation another approach consists in defining discrete difference operators, mimicking the properties of the PDE operators. Again, the subsequent mathematical treatment depends on eigenvector-eigenvalue methods and their properties, see \cite{ChungSpectral} for a deep discussion, followed by the construction of Green's functions, see \cite{DiffElastGraphs} as an example. Yet an intermediate approach addresses time-evolution problems using discrete models for space, as imposed by the graph itself, and continuous evolution in time under the hypothesis that the underlying combinatorial structure of the graph remains stationary, see \cite{DiffElastGraphs} for this view. In contrast with the previous achievements the present work preserves the continuous definitions for the operators in the PDE, and adapt the domain associated with the graph allowing for the use of weak variational formulation results. In that sense this article provides a dual approach to the previous results, however our main motivation is to attain upscaling criteria for highly complex networks which is a strong necessity as previously discussed. Next, we describe the model introduced in this paper.  

%Authors, as revised in \cite{GraphsPDE}, commonly choose the embedding of the graph in $\R^N$ to be the domain. This definition guarantees that any graph will have a domain, moreover the construction of such domain is intuitive and relatively simple to attain. On the other hand, defining the differential operators in the PDE has proven to be a more difficult task since graphs are discrete by nature whereas PDEs can be uderstood as continuously varying. Many authors find different operators on functions that mimic the properties of the operators they want to define, hence results can be proven by adapting the concepts in continuous theory. As discussed, investigation has centered on adapting the definition of the PDE in order to be compatible with a very simple definition for the domain associated with the graph. 

Each clustered region will be modeled by an open bounded, simply connected set in $\R^{N}$ where every point will represent an individual/molecule. In order to approximate the clustering, we propose that every point will exchange these quantities with every element in a small neighborhood. Following this concept if two universes exchange the quantities in question then, the sets representing them will share a non-negligible boundary. Moreover, according to \cite{ComplexNetworksMarteenVanSteen} the highly clustered regions are the ideal medium for rapid communication between the nodes. Therefore, we propose that the observed quantities can be realized as a fluid flow phenomenon in the modeling open set. 
 %It is a well-known fact that behavior of fluids can be accurately described in terms of differential equations, hence 
 For simplicity we adopt the stationary, saturated, Darcy Flow model \eqref{Eq porous media strong} to approximate the exchange of the aforementioned quantities within the region
 %on a bipartite simple graph the behavior of the quantities by a PDE in the interior of each open set and appropriate fluid exchange conditions and boundary conditions, to describe the exchange among universes. 
 %In a sense we are proposing to solve a PDE upon a graph. In this article we will use the Darcy Flow Model, presented below, to approximate the exchange of the quantities mentioned earlier on a bipartite simple graph.
%
%
\begin{subequations}\label{Eq porous media strong}
	\begin{equation}\label{Eq Darcy Strong}
	a\,\u + \grad p + \g = 0\,,
	\end{equation}
	\begin{equation}\label{Eq conservative strong}
	\div \u - F = 0\,\quad \mathrm{in}\; \Omega \, .
	\end{equation}
	%
%	\begin{equation}\label{Eq Drained Condition}
%	p = 0 \,\quad \mathrm{on}\; \partial \Omega \, .
%	\end{equation}
%
\end{subequations}
Here, $a$ is a positive coefficient describing the resistance to the flow of the medium. 
Additionally, coupling conditions will be introduced to describe the exchange between universes, as well as boundary conditions for the multiverse overall behavior. This is a PDE problem defined on a domain associated to a graph, however in order to model successfully the exchange of two quantities simultaneously, the available tools of analysis \cite{MoralesShow2} demand the underlying graph to be bipartite; this additional hypothesis will be necessary and included in Section \ref{Sec PDE Model}. Finally, the coefficient $a$ in \eqref{Eq Darcy Strong} can be interpreted as resistance to the flow within a network e.g. fees and taxes slowing down the exchange of goods and services in a nation state, paradigms impeding to permeate new ideas in a social network \cite{Jackson}, band width limiting the diffusion of information through the Internet \cite{ComplexNetworksMarteenVanSteen}, etc.

The paper is organized as follows. In Section \ref{Sec Preliminaries} we list the results and concepts needed for the exposition. Section \ref{Sec Graph Domain} defines the types of domain to be associated to the graph and proves their existence. Section \ref{Sec PDE Model} introduces the PDE model together with the necessary geometric associated notions, it also shows the formulation of the problem, proves its well-posedness and recovers the strong form. Finally, Section \ref{Sec Conclusion} presents the final discussion, and future work.  
%
%
%%%%%%%%%%%%%%%%%%%%%%%%%%%%%%%%%%%%%%%%%%%%%%%%%%%%%%%%
\section{Preliminaries from Graph Theory and PDE}\label{Sec Preliminaries}
%%%%%%%%%%%%%%%%%%%%%%%%%%%%%%%%%%%%%%%%%%%%%%%%%%%%%%%%
%
%
%
%
%
%
%%%%%%%%%%%%%%%%%%%%%%%%%%%%%%%%%%%%%%%%%%%%%%%%%%%%%%%%
\subsection{Preliminaries from Graph Theory}
%%%%%%%%%%%%%%%%%%%%%%%%%%%%%%%%%%%%%%%%%%%%%%%%%%%%%%%%
%
% 
%
We begin this section with the basic, necessary definitions from graph theory \cite{BondyMurty}.
\begin{definition} \label{Def walk, path and cycle, length}
Let $G = (V, E)$ be a graph
\begin{enumerate}[(i)]
%\begin{enumerate}[label= (\roman*), leftmargin=15pt]
%\begin{asparaenum}[(i)]
\item
The \textbf{degree} of a vertex $v$, is the number of edges that have an endpoint at $v$.
\item
A \textbf{walk} in $G$ from vertex $v_{0}$ to vertex $v_{j}$ is an alternating sequence
\begin{equation*}
%W \defining 
\langle v_{0}, e_{1}, v_{1}, e_{2}, \ldots, v_{n-1}, e_{n}, v_{n}\rangle ,
\end{equation*}
of vertices and edges such that the endpoints of the edge $e_{i}$ are $v_{i-1}$ and $v_{i}$ for all $i = 1, 2, \ldots, n$.

\item
A \textbf{path} is a walk with no repeated edges and no repeated vertices, except possibly the initial and final vertices.

\item
A walk or path is \textbf{trivial} if it has only one vertex and no edges.

\item
A \textbf{cycle} is a non-trivial closed path i.e. it starts and ends on the same vertex.
\end{enumerate}
\end{definition}
\begin{definition}\label{Def simple graph}
A \textbf{self-loop} is an edge that joins a single vertex with itself. A \textbf{multi-edge} is a collection of two or more edges joining identical vertices. A \textbf{simple graph} has neither self-loops nor multi-edges.
\end{definition}
%
%
%\begin{definition}\label{Def vertex degree}
%The \textbf{degree} (or \textbf{valence}) of a vertex $v$ in a graph $G$, denoted $\deg (v)$ is twice the number of self loops plus the number of proper edges joining $v$ with any other vertex of $G$.
%\end{definition}
%
%
%\begin{definition} A \textbf{cycle graph} $C$ is a single vertex with a self-loop or a \textbf{simple graph} whose number of vertices equals its number of edges and can be drawn so that all its vertices and edges lie on a single circle. A cycle graph having $j$ vertices will be denoted by $C_{j} = \langle v_{1}, v_{2}, \ldots, v_{j}, v_{1}\rangle$.
%\end{definition}
%
%
\begin{definition}\label{Def connected graph}
\begin{enumerate}[(i)]
\item A graph is \textbf{connected} if for every pair of vertices $u$ and $v$ there is a walk from $u$ to $v$.

\item An edge $e$ is a \textbf{bridge}, if the graph $G - e$ is not connected. 

\end{enumerate}
\end{definition}
\begin{definition}\label{Def trees and forests graph}
A graph with no cycles is a \textbf{forest}, if additionally the graph is connected it is said to be a \textbf{tree}. In a tree, a vertex of degree one is said to be a \textbf{leaf}.
\end{definition}
The following is a well-known result about trees, \cite{BondyMurty}.
\begin{proposition}
A tree with at least one edge has at least two leaves.	
\end{proposition}
%
%
%
%\begin{definition}\label{Def leaves of a tree}
%	A graph with no cycles is a \textbf{forest}, if additionally the graph is connected it said to be a \textbf{tree}.
%\end{definition}
%%
\begin{definition}\label{Def bipartite graph}
A \textbf{bipartite graph} is a graph whose vertex set $V$ can be partitioned in two subsets $U, W$ such that each edge of $G$ has one endpoint in $U$ and one endpoint in $W$. The pair $U, W$ is called a (vertex) bipartition of $G$, and $U$ and $W$ are called the bipartition subsets.
\end{definition}
Next we recall several definitions about planar graphs \cite{BondyMurty}. 
%
%
% \begin{definition}\label{Def curves}
% Let $\gamma:[0,1]\rightarrow \R$ be a continuous function, then
% %
% \begin{enumerate}[(i)]
%    \item The image set $\gamma([0,1])$ is said to be a \textbf{curve}.
%    
%    \item If $\gamma(0)= \gamma(1)$ then $\gamma([0,1])$ is said to be a \textbf{closed} curve.
%    
%    \item If the function $\gamma$ is one-to-one, then $\gamma([0,1])$ is said to be a \textbf{simple} curve.
% \end{enumerate} 
% %
% In the following, the curve $\gamma([0,1])$ will be denoted by $\gamma$.
% \end{definition}
% %
% %
% Next, we recall a well-known basic result from topology.
% %
% %
% \begin{theorem}[The Jordan Curve Theorem]\label{Th Jordan Curve Theorem}
% Any simple closed curve $\gamma$ in the plane partitions the rest of the plane into two disjoint arcwise-connected open sets. 
% \end{theorem}
% %
% %
% \begin{theorem}\label{Th faces of plane graphs}
% In a nonseparable plane graph other than $K_{1}$ and $K_{2}$, each face is bounded by a cycle. 
% \end{theorem}
%
%
\begin{definition}\label{Def planar graph}
\begin{enumerate}[(i)]
\item A graph is said to be \textbf{embeddable} in the plane, or \textbf{planar}, if it can be drawn in the plane so that its edges intersect only at their ends. 

\item A planar embedding of a graph will be referred to as a \textbf{plane graph}. 

\item A plane graph $G$ partitions the rest of the plane into a number of arcwise-connected open sets. These sets are said to be the \textbf{faces} fo $G$. 

\item We say that a vertex $v$ of a plane graph G is an \textbf{outer vertex}, if it belongs to the boundary of the outer face of $G$.
\end{enumerate}
\end{definition}
\begin{definition}\label{Def dual planar graph}
Let $G$ be a plane graph, 
\begin{enumerate}[(i)]
\item 
The \textbf{dual graph} $G^{*}$ is defined as follows. Corresponding to each face $f$ of $G$ there is a vertex $f^{*}$ of $G^{*}$ and corresponding to each edge $e$ of $G$ there is an edge $e^{*}$ of $G^{*}$. Two vertices $f^{*}$ and $g^{*}$ are joined by the edge $e^{*}$ in $G^{*}$ if and only if their corresponding faces $f$ and $g$ are separated by the edge $e$ in $G$. 

\item 
The \textbf{plane dual} of the plane graph $G$ is a natural embedding of $G^{*}$ in the plane. It is obtained by placing a vertex $f^{*}$ in the corresponding face $f$ of $G$, and then drawing an edge $e^{*}$ in such a way that it crosses the corresponding edge $e$ of $G$ exactly once and crosses no other edge of $G$. We refer to such a drawing as plane dual of the plane graph $G$.   

\end{enumerate}
\end{definition}

\begin{definition}
	A \textbf{curve} in $\R^{2}$ or $\R^{3}$ is the continuous image of a closed interval, we say that a curve is \textbf{simple} if it does not intersect itself.
	%There is no need for the function to be injective since the curve is the image.
\end{definition}
We close this section recalling two well-known results \cite{BondyMurty, Wilson}.
\begin{theorem}\label{Th plane connected graphs characterization}
A plane graph $G$ is connected if and only if it is isomorphic to its double dual $G^{**}$. 
\end{theorem}
\begin{theorem}\label{Th embeddin in 3-d of any graph}
%Let $G = (V, E)$ be a graph, then it 
Every finite graph is embeddable in $\R^{3}$.  
\end{theorem}
%
%
%
%
%%%%%%%%%%%%%%%%%%%%%%%%%%%%%%%%%%%%%%%%%%%%%%%%%%%%%%%%
\subsection{Preliminaries from PDE}
%%%%%%%%%%%%%%%%%%%%%%%%%%%%%%%%%%%%%%%%%%%%%%%%%%%%%%%%
%
%
We start this section introducing the general notation. In the present work vectors are denoted by boldface letters as are vector-valued functions and corresponding function spaces. The symbols $\grad$ and $\div$ represent the gradient and divergence operators respectively. The dimension is indicated by $N$ which will be equal to $2$ or $3$ depending on the context. Given a function $f: \R^{N}\rightarrow \R$ then  $\int_{\mathcal {M} } f\,dS$ denotes the integral on the $N-1$ dimensional manifold $\mathcal{M}\subseteq \R^{\! N}$. Analogously, $\int_{A} f\, d\x$ stands for the integral in the set $A\subseteq \R^{\! N}$; whenever the context is clear we simply write $\int_{A} f$. The symbol $\outer$ denotes the outwards normal vector on the boundary of a given domain $\mathcal{O}\subseteq \R^{N}$. Given an open set $M$ of $\R^{N}$, the symbols $\Vert\cdot\Vert_{0,M}$,  $\Vert\cdot\Vert_{1,M}$, $\Vert\cdot\Vert_{1/2, \partial M}$, $\Vert\cdot\Vert_{-1/2,\partial M}$ and $\Vert\cdot\Vert_{\Hdiv(M)}$ denote the $L^{2}(M)$, $H^{1}(M)$, $H^{1/2}(\partial M)$, $H^{-1/2}(\partial M)$ and $\Hdiv(M)$ norms respectively, while $\vert M\vert$ represents the Lebesgue measure of $M$ in $\R^{2}$ or $\R^{3}$ depending on the context.

Next, we present the general abstract problem to be studied in this article. Let $\X$ and $\Y$ be Hilbert spaces and let $\A: \X\rightarrow \X $$'$, $\B: \X\rightarrow \Y $$'$ and $\C: \Y\rightarrow \Y $$'$ be continuous linear operators, we are to study the following problem
\begin{equation}\label{Pblm operators abstrac system}
%
%\begin{equation*}
%\text{Find a pair}\; (\u,\,p)\in \V\times Q\;\mathrm{sastisfying:}
%\end{equation*}
%
\begin{split}
%\begin{equation}\label{Pblm operators abstrac eqn 1}
\text{Find a pair}\; (\x, \y)\in \X\times \Y: \quad 
\A \x + \B '\y  = F_{1}\quad \text{in}\; \X ' , \\
%\end{equation}
%
%\begin{equation}\label{Pblm operators abstract eqn 2}
- \B \x  + \C \y = F_{2} \quad \text{in}\; \Y ' .
%\end{equation}
\end{split}
\end{equation}
Here $F_{1}\in \X '$ and $F_{2} \in \Y '$. Several variations of systems such as the above have been extensively studied, we present below a well-known result \cite{GiraultRaviartFEM} to be used in this work.
\begin{theorem}\label{Th well posedeness mixed formulation classic}
Assume that the linear operators $\A: \X\rightarrow \X'$, $\B: \X\rightarrow \Y '$, $\C: \Y\rightarrow \Y '$ are continuous and
\begin{enumerate}[(i)]
\item $\A$ is non-negative and $\X$-coercive on $\ker (\B)$. %with coerciveness constant $a$.

\item $\B$ satisfies the inf-sup condition 
\begin{equation}\label{Ineq general inf-sup condition}
   \inf_{\y \, \in \, \Y} \sup_{\x \, \in \, \X}
   \frac{\vert  \B\x(\y) \vert }{\Vert \x\Vert_{\X}\, \Vert \y \Vert_{\Y}}  >0 \, .
\end{equation}

\item $C$ is non-negative symmetric.
\end{enumerate}
Then for every $F_{1} \in \X '$ and $F_{2} \in \Y '$ the problem \eqref{Pblm operators abstrac system}
has a unique solution in $(\x, \y)\in \X \times \Y$, which satisfies the estimate
\begin{equation} \label{mix-est}
\Vert\x\Vert_{\X} + \Vert \y\Vert_{\Y} \leq c\, (\Vert F_{1}\Vert_{\X '} + \Vert F_{2}\Vert_{\Y '}).
\end{equation}
%
%for a positive constant $c$ depending only on $a$, $b$, $\Vert \A\Vert$ and $\Vert \C\vert$.
\end{theorem}
\section{The Graph Domain}\label{Sec Graph Domain}
% % % % % % % % % % % % % % % % % % % % % % % % % % % % % % % % % % % % % % % % % % % % %
%
%
This section is aimed to the construction of a particular topological domain for a given plane graph. The domain must be suitable for setting a PDE problem. To this end, we introduce two paramount definitions of domains associated to graphs depicted below
\begin{figure}[h] %\label{Fig Stream Lines}
	\centering
	\begin{subfigure}[Downscaling Map. ]
		{\resizebox{5.3cm}{5.3cm}
			{\includegraphics{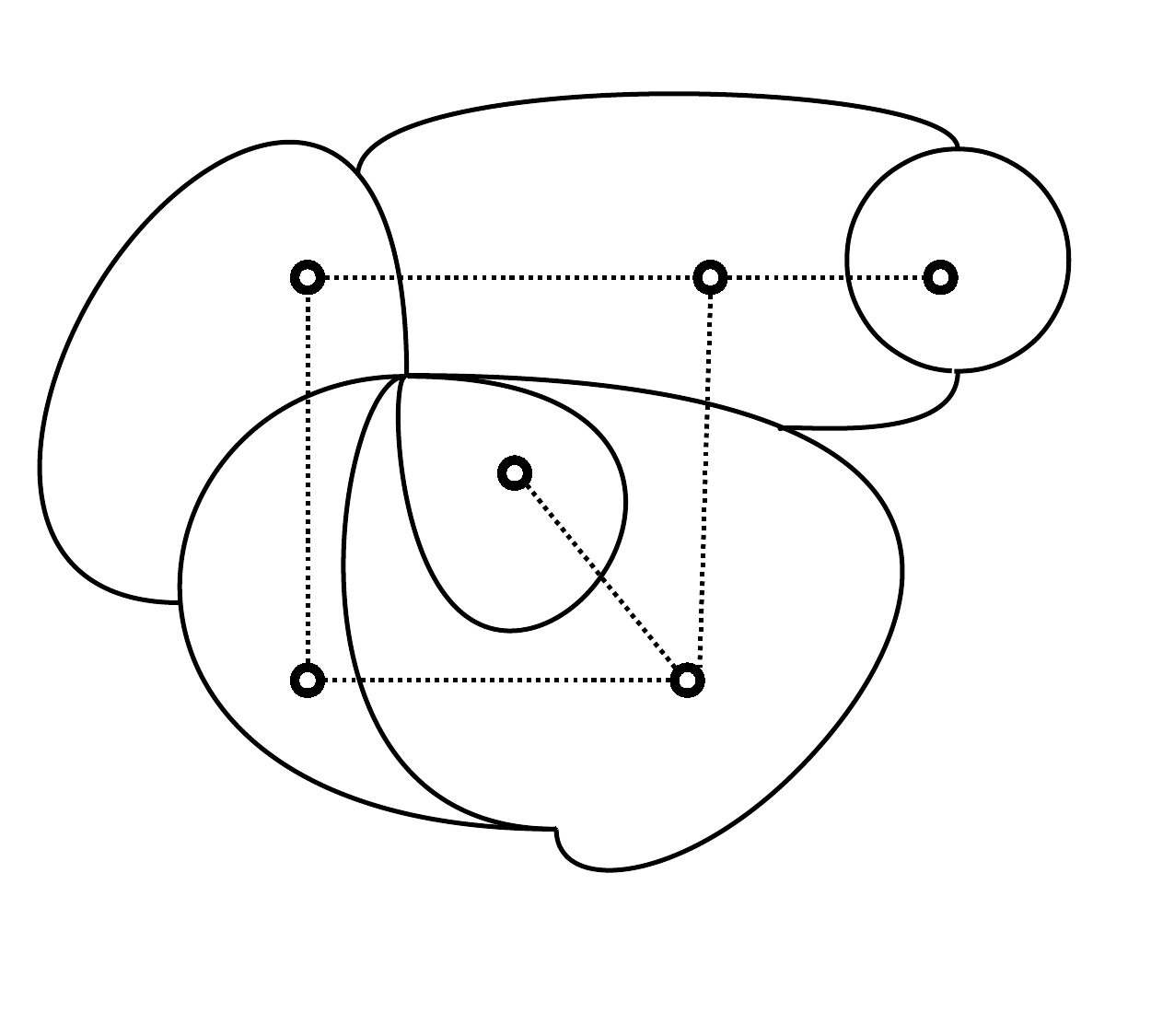} } }
	\end{subfigure} %\label{Fig Downscaling Map}
	\qquad
	~ %add desired spacing between images, e. g. ~, \quad, \qquad etc.
	%(or a blank line to force the subfigure onto a new line)
	\begin{subfigure}[Tubular $\epsilon$-Map.]
		{\resizebox{5.3cm}{5.3cm}
			{\includegraphics{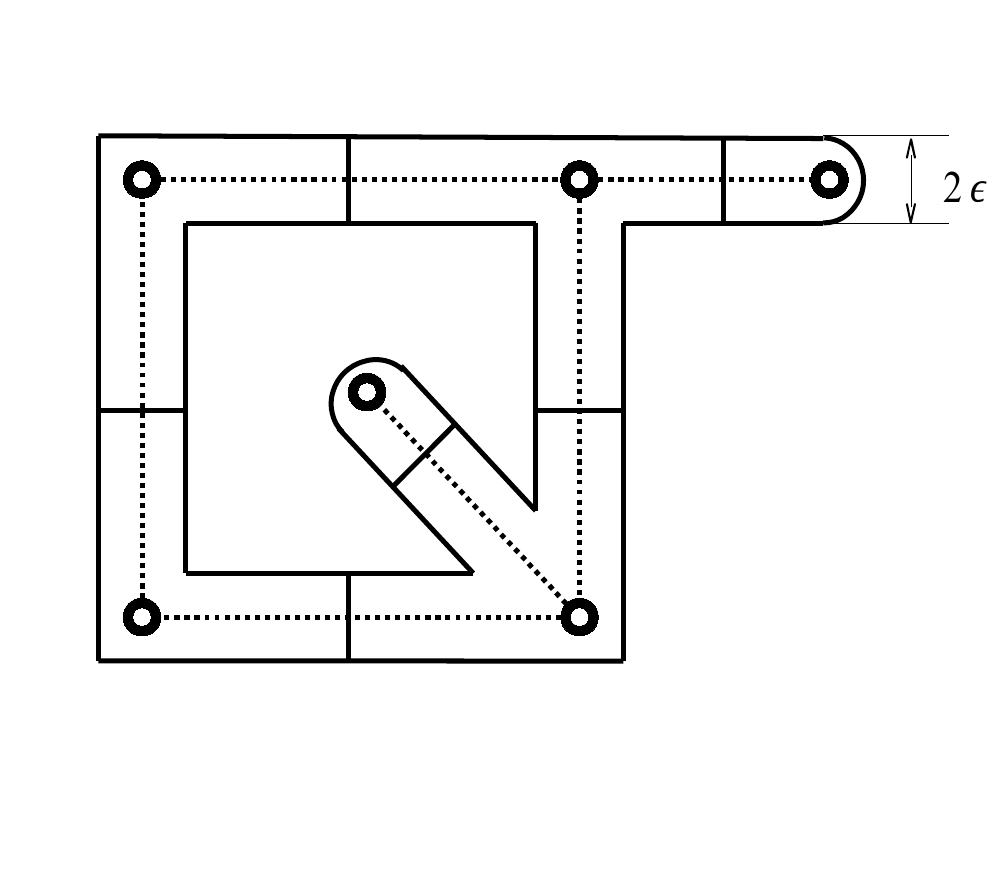} } }                
	\end{subfigure} %\label{Fig Tubular Map} 
	%
	%
	%        \centering
	%        \begin{subfigure}[Downscaling Map. ]
	%                {\resizebox{7cm}{7cm}
	%{\includegraphics{DownscalingMapFinal.pdf} } }
	%        \end{subfigure} %\label{Fig Downscaling Map}
	%        \qquad
	%        ~ %add desired spacing between images, e. g. ~, \quad, \qquad etc.
	%          %(or a blank line to force the subfigure onto a new line)
	%          \begin{subfigure}[Tubular Map.]
	%                {\resizebox{7cm}{7cm}
	%{\includegraphics{TubularMapFinal.pdf} } }                
	%        \end{subfigure} %\label{Fig Tubular Map} 
	%
	\caption{Figure (a) depicts a downscaling map example for a given plane graph represented in dotted line. Figure (b) depicts a tubular map example for the same given plane graph. \label{Fig Downscaling and Tubular Maps} }
\end{figure}
%
%
%
%\begin{definition}\label{Def Boundary of Plane Graphs}
%Let $G = (V, E)$ be a plane graph, we say that the $v\in V$ is an \textbf{outer vertex} if it belongs to the boundary of the outer face of $G$. 
%\end{definition}
%
%
\begin{definition}\label{Def existence of tubular domain}
	Let $G = (V, E)$ be a connected graph embedded in either $\R^{2}$ or $\R^{3}$ such that its edges are simple curves. Denote by $\tilde{v}$, $\tilde{e}$ the points and lines representing the vertices and edges of $G$ and $\tilde{G} \defining (\bigcup \{\tilde{v}: v\in V \} , \bigcup \{\tilde{e}: e\in E \} )$. Let $\epsilon> 0$ be such that the collection of balls $\{B(\tilde{v}, \epsilon): v\in V\}$ is pairwise disjoint. Define 
	\begin{enumerate}[(i)]
		\item The \textbf{tubular} $\boldsymbol{\epsilon}$-\textbf{region} 
		\begin{equation}\label{Eq Tubular Domain first stage}
		\U_{G}^{\epsilon} \defining \{x: d(x, \tilde{G}) < \epsilon \} .
		\end{equation}
		\item For each edge $e \in E$ choose a smooth simple curve $e'$, approximating $\tilde{e}$. Denote by $\ell_{e}$ a secant line (or secant plane) through the midpoint of $e'$ and let $C_{e}$ be the connected component of $\U_{G}^{\epsilon}\cap \ell_{e}$ containing such midpoint.

		\item Let $\{C_{e}: e\in E \}$ be as defined above. For each $v\in V$ and $w$ adjacent to $v$, let $\U_{v,w}^{\epsilon}$ be the tubular $\epsilon$-region corresponding to the induced subgraph $G_{v,w} \defining \big(\{v, w\}, vw \big)$. The set $C_{vw}$ divides $\U_{v,w}^{\epsilon}$ into two open regions, one containing $\tilde{v}$ and one containing $\tilde{w}$ denoted by $H(v, w)$ and $H(w, v)$ respectively. Define the \textbf{starred} region of $v$ by
		\begin{equation}\label{Eq Tubular Domain local}
		\U_{v}^{\epsilon} \defining \bigcup_{w\,\in\, V:  \; vw \,\in \, E}H(v, w).
		\end{equation}
		\item  The collection $\{\U_{v}^{\epsilon}: v\in V \}$ is said to be a \textbf{tubular} $\boldsymbol{\epsilon}$-\textbf{map}, or simply a tubular map, of the graph $G$.
	\end{enumerate}
\end{definition}

Now, we introduce the concept of downscaling map.

\begin{definition}\label{Def Maps of Graphs}
Let $G = (V, E)$ be a plane connected graph, we will say that a \textbf{downscaling map} of $G$ is a collection of bounded open sets $\{\O_{v}: v\in V \}$ called \textbf{regions}, such that
\begin{enumerate}[(i)]
\item $v\in \O_{v}$ for all $v\in V$.

\item If $v\neq w$ then $\O_{v}\cap \O_{w} = \emptyset$.

\item $\O_{v}$ is simply connected for all $v\in V$.

\item\label{Char Global Simply Connectedness} The set $\O$ defined by
\begin{equation}\label{Def Downscaling Map Domain}
\O \defining \interior\left(\cl \bigcup_{v \, \in \, V} \O_{v} \right) 
\end{equation}
is simply connected. We define $\O$ as the \textbf{domain} of the downscaling map.

\item Two elements of the collection share non-negligible boundary if and only if the vertices they contain are connected in the graph, i.e. $\vert\partial \O_{v} \cap \partial \O_{w}\vert > 0$ if and only if $vw\in E$.

\item\label{Char Coast Condition} If $v$ is an outer vertex then $\vert\partial \O_{v}\cap \partial \O \vert>0$.

\end{enumerate}
Finally, we will say that the \textbf{regularity} of the map is given by the lowest degree of regularity of its elements. 
\end{definition}
\begin{remark}
	 Notice the following
\begin{enumerate}[(i)]
\item A tubular map satisfies all the conditions of a downscaling map, except possibly for the global simply connectedness condition (Definition \ref{Def Maps of Graphs}\eqref{Char Global Simply Connectedness}).

\item A tubular map of a plane graph defines a downscaling map if and only if the graph 
is a tree. 

%\item Since any graph is embeddable in $\R^{3}$ an analogous procedure is possible in 3-d for the embedding of a nonplanar graph. 
\end{enumerate} 
\end{remark}
The next two results are central in proving the existence of a downscaling map for a simple, plane, connected graph. %In order to better illustrate the technique the reader is invited to see 
The intuitive idea and technique are depicted in \textit{Figure} \ref{Fig Constructing Downscaling Maps} (b).  
\begin{lemma}\label{Th downscaling map, graph no outer trees}
Let $G = (V, E)$ be a connected, simple, plane graph such that no bridges are in the boundary of its outer face. Then, there exists a downscaling map for $G$. Moreover, this existence can be attained for any chosen level of regularity.
\end{lemma}
\begin{proof}
 If $G$ has no edges then it must be a single vertex $v$, thus an open ball centered at $v$ will satisfy the definition of downscaling map. We will henceforth assume that $G$ has at least one edge.
 
 Let $G^{*}$ be the plane dual graph of $G$ drawn with disjoint simple curves as edges. The regions defined by the faces $f^{*}$ of $G^{*}$ are the natural candidate to define a downscaling map of $G$, however they fail because of two reasons. On one hand, according to \textit{Theorem} \ref{Th plane connected graphs characterization} we know that the double dual $G^{**}$ is isomorphic to $G$. In particular, the outer face of $G^{*}$ must contain a vertex of the plane graph $G$. On the other hand, given a face $f^{*}$ of the plane dual containing an outer vertex $v$, it would not necessarily hold that $\vert \partial f^{*} \cap \partial f^{*}_{0} \vert > 0$, where $f_{0}^{*}$ is the outer face of $G^{*}$, as demands the condition \eqref{Char Coast Condition} in \textit{Definition} \ref{Def Maps of Graphs}.

We overcome the first deficiency as follows, let $v_{0}$ be the unique vertex in $G$ contained in the outer face $f_{0}^{*}$ of $G^{*}$ and let $C$ be the cycle in $G$ bounding its outer face $f_{0}$. Clearly $v_{0}$ belongs to $C$, now let $\epsilon>0$ be such that the tubular $\epsilon$-region of the induced graph $C - v_{0}$ is completely contained in $(f_{0}^{*})^{c}$. Define the region 
\begin{equation*}
\O_{0} \defining  f_{0}^{*}\cap \bigcup_{ x \, \in \, (f_{0})^{c} } B(x, \epsilon).
\end{equation*}
Since $f_{0}$ is the outer face, it is clear that $f_{0}^{c}$ and $\bigcup_{ x \, \in \, (f_{0})^{c} } B(x, \epsilon)$ is simply connected. Therefore $\O_{0}$ is open bounded, simply connected and contains $v_{0}$. For the second deficiency, let $x_{0}$ be the point representing the outer face $f_{0}$ in the plane dual $G^{*}$. Now let $\delta > 0 $ be such that $\cl B(x_{0}, \delta)\subset f_{0}$, then the collection 
\begin{equation*}
\{f^{*} - \cl B(x_{0}, \delta): f^{*} \; \text{face of} \; G^{*}\} \cup \{ \O_{0} - \cl B(x_{0}, \delta) \} , 
\end{equation*}
constitutes a downscaling map for the graph $G$.
 
Finally, since the smooth curves are dense in the plane, it is clear that the boundaries of the elements of this downscaling map can be continuously deformed, to define a new downscaling map with any required level of regularity. 
\qed
\end{proof}
\begin{figure}[h] %\label{Fig Stream Lines}
        \centering
        \begin{subfigure}[Starting Graph $G$. ]
                {\resizebox{5.3cm}{5.3cm}
{\includegraphics{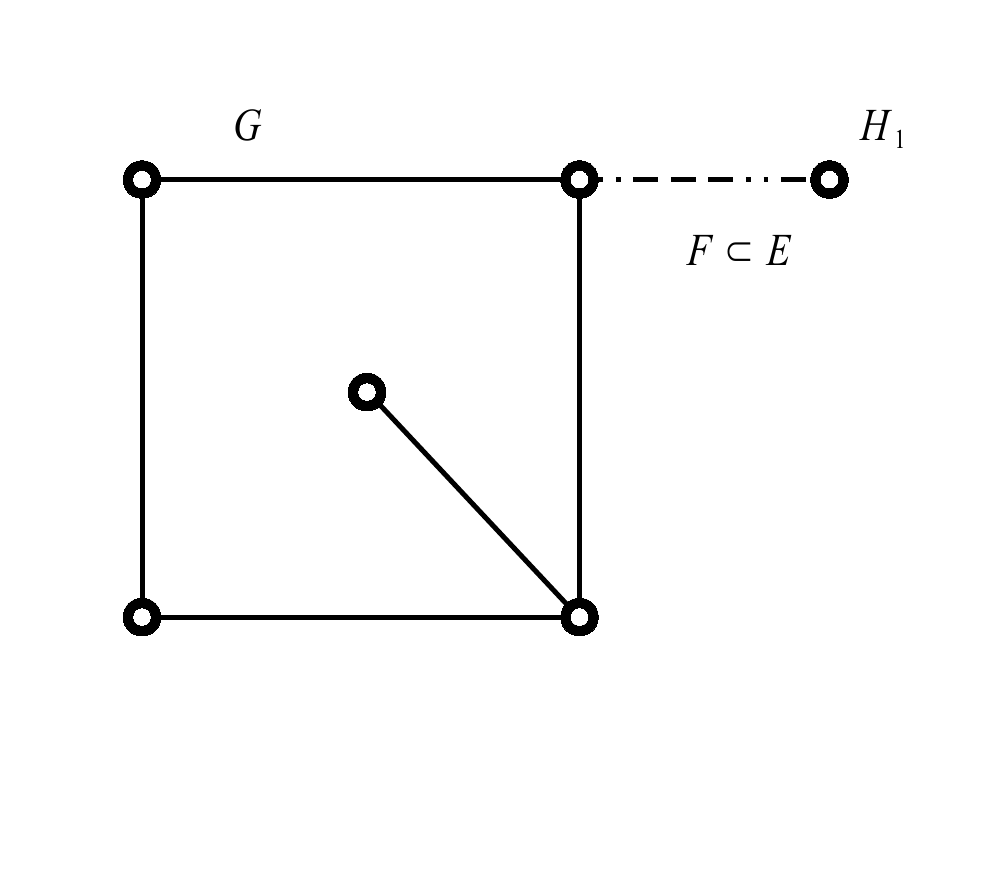} } }
        \end{subfigure} %\label{Fig Downscaling Map}
        \qquad
        ~ %add desired spacing between images, e. g. ~, \quad, \qquad etc.
          %(or a blank line to force the subfigure onto a new line)
          \begin{subfigure}[Dual Graph of $H = G (V,E- F)$.]
                {\resizebox{5.3cm}{5.3cm}
{\includegraphics{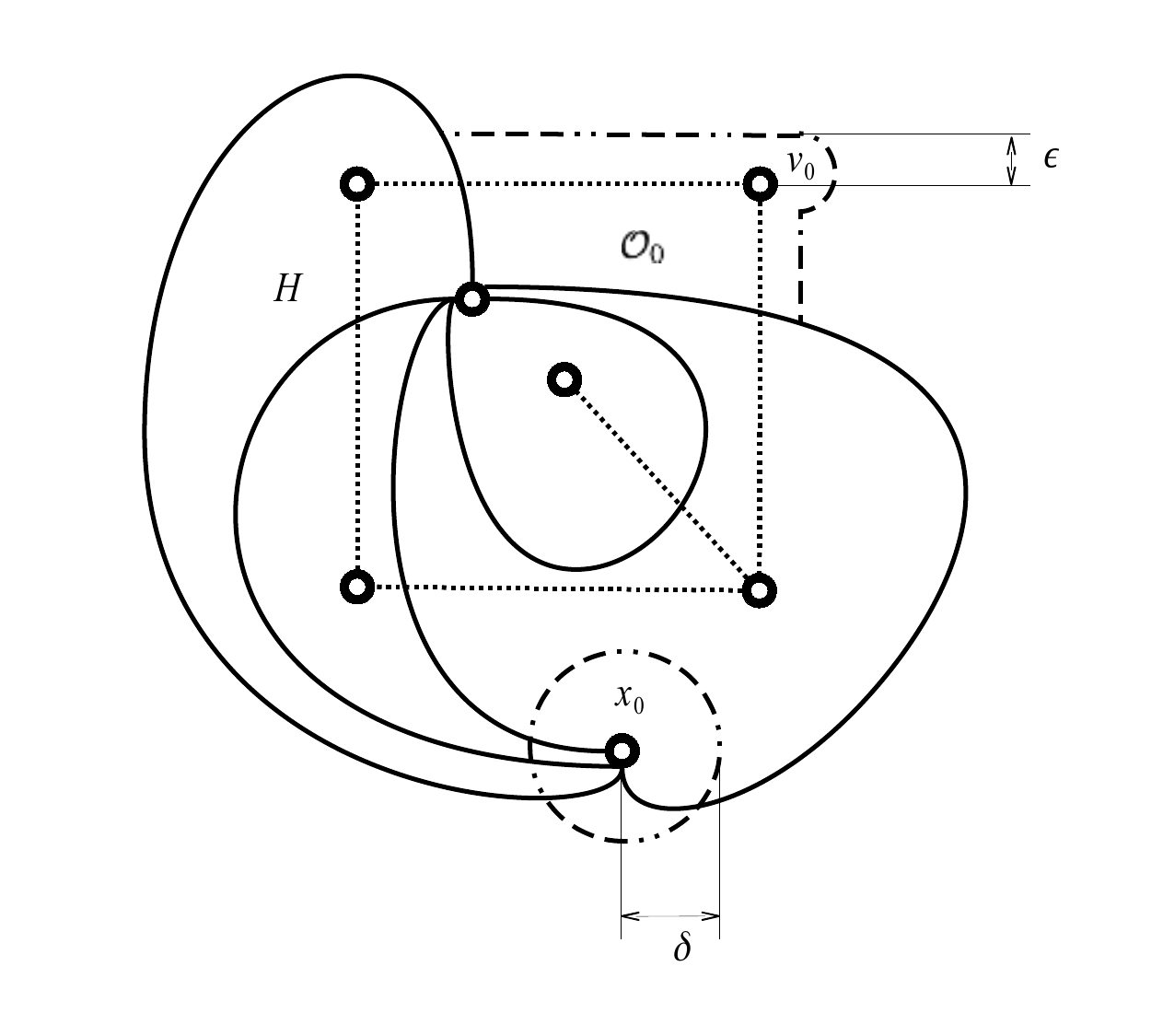} } }                
        \end{subfigure} %\label{Fig Tubular Map} 
%
%
%        \centering
%        \begin{subfigure}[Starting Graph $G$. ]
%                {\resizebox{7cm}{7cm}
%{\includegraphics{GraphFinal.pdf} } }
%        \end{subfigure} %\label{Fig Downscaling Map}
%        \qquad
%        ~ %add desired spacing between images, e. g. ~, \quad, \qquad etc.
%          %(or a blank line to force the subfigure onto a new line)
%          \begin{subfigure}[Dual Graph of $H = G(V, E- F)$..]
%                {\resizebox{7cm}{7cm}
%{\includegraphics{DualGraphFinal.pdf} } }    
%
%            
%        \end{subfigure} %\label{Fig Tubular Map} 
\caption{Figure (a) depicts a starting graph $G$ and its set of outer bridges $F$, it also shows in dashed line the tree $H_{1}$ removed from $G$, as in the proof of \textit{Theorem} \ref{Th downscaling map general}. Figure (b) depicts the construction seen in \textit{Lemma} \ref{Th downscaling map, graph no outer trees} for the graph $H$ depicted in dotted line. The construction of the domain $\O_{0}$ and the removal of the ball $B(x_{0}, \delta)$ are depicted in dashed dotted line. \label{Fig Constructing Downscaling Maps} }
\end{figure}
\begin{proposition}\label{Prop. Trees downscaling map}
	Let $G$ be a connected, simple, plane graph, let $F$ be the set of bridges in the outer face of $G$, and $H_{1},...,H_{k}$ be the connected components of the graph defined by removing the edges in $F$. Let $\tilde{G}$ be the graph whose vertex set is $\{v_{1},...,v_{k}\}$, where  $v_{i}$ is connected to $v_{j}$ if and only if $H_{i}$ is connected to $H_{j}$ by an edge in $F$. Then, $\tilde{G}$ is a tree.
\end{proposition}
\begin{proof}
	If $\tilde{G}$ contains a cycle, the removal of one edge in such cycle does not disconnect $\tilde{G}$. Hence such edge is not a bridge contradicting the definition of $F$.
	\qed
\end{proof}
Finally, we present the main result of this section.
\begin{theorem}\label{Th downscaling map general}
Let $G = (V, E)$ be a connected, simple, plane graph. Then, there exists a downscaling map with smooth boundaries for $G$. 
\end{theorem}
\begin{proof}
First, if $G$ has no cycles then it is a tree and its tubular map constitutes a downscaling map. Hence, from now on we assume that $G$ has at least one cycle.

Let $H_{1},...,H_{k}$ be the components and $\tilde{G}$ be its associated graph as defined in Proposition \ref{Prop. Trees downscaling map}. We will proceed by induction over $k$. If $k=1$, then Lemma \ref{Th downscaling map, graph no outer trees} provides the required downscaling map. If $k>1$, then renumbering the components if necessary, suppose that $H_{k}$ is such that $v_{k}$ in $\tilde{G}$ is a leaf. Let $\Lambda$ be a downscaling map with smooth boundaries for $H_{k}$ and let $\M$ be its domain. Analogously, let $\Theta$ be a downscaling map with smooth boundaries for the graph such that includes $H_{1},...,H_{k-1}$ together with every bridge connecting these components and let $\O$ be its domain. By deforming their boundaries, it can be additionally assumed that $d(\O, \M)>0$, while preserving the smoothness of their boundaries. 

Let $v$ and $u$ be the vertices such that $vu$ is the unique bridge connecting $H_{k}$ with some component $H_{i}$ with $1\leq i < k$. Clearly $v \in H_{k}$ and $u\in H_{i}$. Denote by $\O_{u}$ and $\M_{v}$ the regions of $\Theta$ and $\Lambda$ containing $u$ and $v$ respectively. Consider the graph $\hat{G}$ consisting of the vertices $u$, $v$ and the edge $uv$. Take two points $p$ and $q$ in the interior of the one dimensional manifolds $\partial\O\cap\partial\O_{u}$ and $\partial \M\cap\partial \M_{v}$ respectively. Choose a simple curve $C$ connecting $u$ and $v$, which intersects $\partial\O$ and $\partial \M$ at the unique points $p$, $q$ and such such that $d(C, \O ')$ and $d(C, \M ')$ are positive; where $\O ' \defining \O - \O_{v} $ and $\M'\defining \M - \M_{v}$.
 %
%\begin{multicols}{2} 
%\begin{align*}
%& \O ' \defining %\interior\bigg(\cl \bigcup_{w\neq u} \O_{w} \bigg) = 
%\O - \O_{v} \,  ,  &
% & \M ' \defining %\interior\bigg(\cl \bigcup_{\tilde{w}\neq v} M_{ \tilde{w} } \bigg) =
%\M - \M_{v} \,  .  
%\end{align*}
%\end{multicols}
%
% 
%Here, $\O_{w}$ and $M_{\tilde{w} }$ are the regions in $\O$ and $M$ containing $w$ and $\tilde{w}$ respectively. 
Choose $\epsilon > 0 $ strictly less than $\mathrm{min} \{d(C,\O '),d(C,\M')\}$ and such that if $\mathcal{G}$ denotes the $\epsilon$-map of $\hat{G}$, then it verifies that $\mathcal{G}\cap \O \subset \O_{u}$ and $\mathcal{G}\cap \M \subset \M_{u}$. Redefine $\O_{u}$,  $\M_{v}$ in order to include the regions in the tubular $\epsilon$-map $\mathcal{G}$, corresponding to $u$ and $v$ respectively. If necessary, deform the boundaries of $\O_{u}$ and $\M_{v}$ to attain the required smoothness and denote the outcome by $O_{u}$ and $M_{v}$ respectively. The collection $\Lambda - \{ \M_{v} \} $, $\Theta - \{ \O_{u} \}$ together with $\{ O_{u} \}$ and $\{ M_{v} \} $ constitutes a downscaling map for $G$ with the required regularity. %\\
\hfill \qed
\end{proof}
%
%
% % % % % % % % % % % % % % % % % % % % % % % % % % % % % % % % % % % % % % % % % % % % %
\section{The Downscaled Bipartite Model}\label{Sec PDE Model}
% % % % % % % % % % % % % % % % % % % % % % % % % % % % % % % % % % % % % % % % % % % % %
%
%
%
%
%
%
%%%%%%%%%%%%%%%%%%%%%%%%%%%%%%%%%%%%%%%%%%%%%%%%%%%%%%%%%%%%%%%%
\subsection{Geometric Setting and Modeling Function Spaces}
%%%%%%%%%%%%%%%%%%%%%%%%%%%%%%%%%%%%%%%%%%%%%%%%%%%%%%%%%%%%%%%%
%
%
In this section we give the geometric setting for the formulation of the problem.
\begin{definition}\label{Def geometric convention WVF}
Let $\Omega$ be a connected bounded region with smooth boundary, let $\triang = \{K: K\in \triang\}$ be a bipartite map and denote by $\tone = \{L: L \in \tone\}$, $\ttwo = \{M: M\in \ttwo\}$ the  bipartition, or \textbf{bi-coloring}, of the map. 
\begin{enumerate}[(i)]
\item
For each region $K\in \triang$ denote by $\outer$ the outer normal vector to its boundary $\partial K$.

\item%\label{Def normal vector convention}
For each region $K\in \triang$ denote by $\n: \partial K \rightarrow \R^{N}$ the ``\textbf{normal}" vector by
\begin{equation}\label{Def normal vector}
\n ( \vec{x} )\defining\begin{cases}
\outer( \vec{x} ) & K\in \tone ,\\
-\outer  ( \vec{x} ) & K \in \ttwo\; \text{and}\;\vec{x} \in \partial K\cap \Omega \, , \\
\outer  ( \vec{x} ) & K \in \ttwo \, \text{and} \; \vec{x} \in \partial K\cap \partial \Omega \, .
\end{cases}
\end{equation}

\item
Define $\displaystyle \Omeone \defining \bigcup \{L:L\in \tone\}$, $\displaystyle\Ometwo \defining \bigcup \{M:M\in \ttwo\}$ .

\item
Define $\displaystyle\Gamma\defining \bigcup\{\partial K: K\in \triang\} - \partial \Omega$.

\item
Define 
\begin{equation}\label{Def interface network or graph edges}
\begin{split}
\edges \defining \bigcup \{\partial L \cap \partial M: L \in \tone, \; M\in \ttwo\} 
& =
\{\sigma: \sigma \;\text{is the interface of two regions of different type} \} \\
& =
\{\sigma: \sigma \;\text{defines an edge in the graph} \; G  \;\text{of the map} \; \triang\} .
\end{split}
\end{equation}
\end{enumerate}
\end{definition}
\begin{remark} Notice that
	\begin{enumerate} [(i)]
	\item Simple connectedness upon the domain is not required in order to include the tubular map.
	
	\item In agreement with the previous section notice that a bipartite map will always be induced by a bipartite simple graph.
	\end{enumerate}
\end{remark}
In order to successfully associate a well-posed problem we endow the Problem \eqref{Eq porous media strong} with boundary conditions \eqref{Eq Drained Condition decomposed}, \eqref{Eq Non-Flux Condition decomposed}, together with the exchange interface conditions of normal flux balance \eqref{Eq normal stress balance} and normal stress balance \eqref{Eq normal stress balance}. This gives the following strong problem %Additionally, in order to find a well-posed variational formulation we must introduce a modification of the original model problem \eqref{Eq porous media strong}.
%\vspace{-10pt}
%
%
\begin{multicols}{2}
\begin{subequations}\label{Eq porous media strong decomposed}
\begin{equation}\label{Eq Darcy Strong decomposed 0}
a\, \uone + \grad \pone 
+ \g 
= 0\,,
\end{equation}
\begin{equation}\label{Eq conservative strong decomposed 0}
\div \uone =  F
\,\quad \mathrm{in}\; \Omeone \, .
\end{equation}
\begin{equation}\label{Eq Drained Condition decomposed}
\pone = 0  \quad \mathrm{on}\; \partial \Omega_{1}\cap \partial \Omega \, .
\end{equation}
\begin{equation}\label{Eq Darcy Strong decomposed 1}
a \, \utwo + \grad \ptwo 
+ \g 
= 0\,,
\end{equation}
\begin{equation}\label{Eq conservative strong decomposed 1}
\div \utwo = F\,\quad \mathrm{in}\; \Ometwo .
\end{equation}
\begin{equation}\label{Eq Non-Flux Condition decomposed}
\utwo\cdot\n = 0 \quad \mathrm{on}\; \partial \Omega_{2}\cap \partial \Omega  ,
\end{equation}
%
%\end{subequations}
%
%
%\end{multicols}
%
%\hspace{-0.2in}together with the exchange interface normal flux and normal stress balance conditions 
%\vspace{-20pt}
%
%
%\begin{multicols}{2}
%
%\begin{subequations}\label{Eq interface balance conditions}
%
\begin{equation}\label{Eq normal flux balance}
\uone \cdot\n - \utwo \cdot\n = 
\beta\, \ptwo + \flux \, , 
\end{equation}
\begin{equation}\label{Eq normal stress balance}
\ptwo - \pone = %\alpha \, \uone\cdot \n + 
\stress\, \quad\mathrm{on}\;\Gamma \, .
\end{equation}
\end{subequations}
\end{multicols}
\begin{hypothesis}\label{Hyp non null local storage coefficient}
   It will be assumed that the storage exchange and the friction coefficients, $\beta: \Gamma\rightarrow [0, \infty)$, $a: \Omega\rightarrow (0,\infty)$ respectively, verify that $\beta\in L^{\infty}(\Gamma)$, $\Vert \beta \, \ind_{\Gamma} \Vert_{L^{1} (\Gamma) }  > 0$ and $a\in L^{\infty}(\Omega)$, $\Vert \frac{1}{a}\Vert_{L^{\infty}(\Omega)}>0 $.
\end{hypothesis}
%
%
%
%
%
%%%%%%%%%%%%%%%%%%%%%%%%%%%%%%%%%%%%%%%%%%%%%%%%%%%%%%%%%%%%%%%%
%\subsection{Preliminary Results and Modeling Function Spaces}
%%%%%%%%%%%%%%%%%%%%%%%%%%%%%%%%%%%%%%%%%%%%%%%%%%%%%%%%%%%%%%%%
%
%
In order to introduce the modeling spaces used in the formulation we first notice that $\{L: L\in \tone \}$ and $\{M: M\in \ttwo \}$ are the simply connected components of $\Omega_{1}$ and $\Omega_{2}$ respectively. Then, 
\begin{align*} %\label{Eq direct sum of strong spaces}
& \Hdiv(\Omega_{1}) = \bigoplus_{L \,\in \, \tone} \Hdiv(L)  , &
& H^{1}(\Omega_{2}) = \bigoplus_{M \,\in \, \ttwo} H^{1}(M) .
\end{align*}
%
%Moreover, if $\H(\Omega_{2})\defining \{\v\in \Hdiv(\Omega_{2}): \div \v = 0 \; \text{in}\, \Omega_{2} , \, \v\cdot \n = 0\; \text{on}\, \partial \Omega_{2} \}$ and $\H(M) \defining \{\v\in \Hdiv(M): \div \v =0 \; \text{in}\, M, \, \v\cdot \n = 0 \; \text{on}\, \partial M \}$ with $M\in \ttwo$ then
%%
%\begin{equation*} %\label{Eq direct sum of pivot spaces}
%\H(\Omega_{2}) = \bigoplus_{M \,\in \, \ttwo} \H(M) . 
%\end{equation*}
%%
%%
%The following space is introduced in order to decouple the action of the traces in the variational formulation
% %
%\begin{equation}\label{Eq Decoupling Trace Statement}
%E(\Omega_{2}) \defining \big\{ q\in H^{1}(\Omega_{2}): 
%q\ind_{\partial M \cap \partial L }\in H^{1/2}(\partial M ) \; \text{for all }\, 
%(L, M) \in\tone\times\ttwo\big\} .
%\end{equation}
% %
%We endow $E(\Omega_{2})$ with the $H^{1}$-norm. Now we define the modeling spaces as
%
%
The following space is introduced in order to couple adequately, the action of the pressures traces in the variational formulation
\begin{equation}\label{Eq Decoupling Trace Statement}
\begin{split}
E(\Omega_{2})  \defining & \big\{ q\in H^{1}(\Omega_{2}): 
q\ind_{\partial M \cap \partial L }\in H^{1/2}(\partial L ) \; \text{for all }\, 
(L, M) \in\tone\times\ttwo\big\}\\
=  & \big\{ q\in H^{1}(\Omega_{2}): 
q\ind_{\Gamma }\in H^{1/2}(\Gamma ) \} .
\end{split} 
\end{equation}
We endow $E(\Omega_{2})$ with the $H^{1}(\Omega_{2})$ inner product. It is direct to see that $E(\Omega_{2})$ is a closed subspace of $H^{1}(\Omega_{2})$ and consequently a Hilbert space. Now define
\begin{equation}\label{Def Integrals on the trace}
\V(\Omega_{2}) \defining  \big\{ \v\in \L^{\! 2}(\Omega_{2}): \vtwo = \grad \qtwo\;
\text{for some}\; \qtwo\in E(\Omega_{2}) \big\} 
= \grad (E(\Omega_{2}) ) ,
\end{equation}
endowed with the $\L^{2}(\Omega_{2})$ inner product. Next we show a necessary result.
\begin{lemma}\label{Th Completeness of V(Omega_2)}
Let $E(\Omega_{2})$ and $\V(\Omega_{2})$ be as defined in \eqref{Eq Decoupling Trace Statement},  \eqref{Def Integrals on the trace} respectively, and define 
\begin{equation}\label{Def Space Generating with Isomorphism}
E_{0} (\Omega_{2})\defining \Big\{ \qtwo \in E(\Omega_{2}): \int_{\Omega_{2}}\qtwo = 0 \Big\} .
\end{equation}
Then,
\begin{enumerate}[(i)]
\item There exists a constant $C>0$ depending only on the domain $\Omega_{2}$ such that
\begin{align}\label{Ineq Control on Space Generating with Isomorphism}
& \Vert \rtwo \Vert_{1, \Omega_{2}} \leq C \Vert \grad \rtwo \Vert_{0, \Omega_{2}} , &
& \text{for all }\; \rtwo\in H\, .
\end{align}

\item The space $\V(\Omega_{2})$ is Hilbert.
\end{enumerate}
\end{lemma}
\begin{proof} 
\begin{enumerate}[(i)]
\item Clearly $E_{0} (\Omega_{2})$ is closed and because of the matching property of traces for elements of $E(\Omega_{2})$ the application $\rtwo\mapsto \Vert \grad \rtwo \Vert_{0, \Omega_{2}}$ is norm in $E_{0} (\Omega_{2})$. Due to the Rellich-Kondrachov Theorem this norm is equivalent to the standard one in $E_{0} (\Omega_{2})$ i.e., there exists $C  > 0$ depending only on the domain $\Omega_{2}$ satisfying the statement \eqref{Ineq Control on Space Generating with Isomorphism}.

\item Clearly, it is only necessary to check that $\V(\Omega_{2})$ is complete. Let $\{\vtwo^{n}: n\in \N \}$ be a Cauchy sequence in $\V(\Omega_{2})$, then there exists a sequence $\{\ptwo^{n}: n\in \N\}$ in $E(\Omega_{2})$ such that $\grad \qtwo^{n} = \vtwo^{n}$, then, the function $\rtwo^{n} \defining \qtwo^{n} - \frac{1}{\vert \Omega_{2}\vert}\int_{\Omega_{2}} \qtwo^{n}$  belongs to $E_{0} (\Omega_{2})$ and $\grad \rtwo = \vtwo$. Due to the previous part it follows that the sequence $\{ \rtwo^{n}: n\in \N\} \subseteq E_{0} (\Omega_{2})$ is Cauchy, consequently it converges to an element $\rtwo\in E_{0} (\Omega_{2})\subseteq E(\Omega_{2})$. Finally, since the gradient map $\grad$ from $E(\Omega_{2})$ onto $\V(\Omega_{2})$ is continuous, the result follows. 
\qed
\end{enumerate}
\end{proof}
Now we are ready to introduce the functional setting of the problem, we have
\begin{subequations}\label{Def function spaces}
\begin{equation}\label{Def spaces of velocities}
\X  \defining 
%\{[\vone, \qtwo]\in \mathbf{L}^{2}(\Omega_{1})\times L^{2}(\Omega_{2}): 
%\div \vone \vert_{L}\in L^{2}(L) \; \text{for all} \; L\in \tone , 
%\grad \qtwo\vert_{M}\in L^{2}(M) \; \text{for all} \; M\in \ttwo\} 
%\big\{[\vone, \qtwo]\in
\Hdiv(\Omega_{1}) \times E(\Omega_{2}) ,
%: \vone\cdot\n\vert_{\Gamma} \in L_{2}(\Gamma) \big\} ,
\end{equation}
\begin{equation}\label{Def spaces of pressures}
%\begin{split}
\Y \defining %\big\{[\vtwo, \qone]\in 
\V(\Omega_{2})\times L^{2}(\Omega_{1}) .
%: \int_{\Omega_{2}} \vtwo\cdot \wtwo = 0 \; \text{for all} \; \wtwo\in \H(\Omega_{2}) 
 %\big\} .
\end{equation}
Endowed with their natural inner product and norms
\begin{equation}\label{Def norm space of velocities}
\big\Vert [\vone, \qtwo] \big\Vert_{\X}\defining
\big\{\Vert \vone \Vert_{\Hdiv(\Omega_{1})}^{2} 
% + \Vert \vone\cdot \n \Vert_{L^{2}(\Gamma)}^{2} 
+ \Vert \qtwo\Vert_{H^{1}(\Omega_{2})}^{2} \big\}^{\tfrac{1}{2}} , 
\end{equation}
\begin{equation}\label{Def norm space of pressures}
\big\Vert [\vtwo, \qone] \big\Vert_{\Y}\defining 
\big\{\Vert \vtwo\Vert_{\L^{2}(\Omega_{2})}^{2} 
+ \Vert  \qone \Vert_{ L^{2}(\Omega_{1})}^{2}\big\}^{\tfrac{1}{2}} .
\end{equation}
\end{subequations}
%
%
%In the following
%
%
\begin{remark}\label{Rem notation of the duality product}
\begin{enumerate}[(i)]
\item Notice that the definition of spaces gathers the functions of high regularity in $\X$ and the functions of low regularity in $\Y$. This choice is made on one hand to satisfy the hypotheses of Theorem \ref{Th well posedeness mixed formulation classic} and, on the other hand, to preserve the remarkable aspect that the underlying modeling spaces $\X$ and $\Y$ are free of coupling conditions. This approach will lead to a version of mixed formulation different from the one presented in \cite{MoralesShow2} and \cite{Morales2}, which shares the coupling-free spaces feature. 

\hspace{-0.3in}In order to avoid heavy notation, in the sequel we adopt the following conventions
\item Let $\Delta$ be an open bounded set, $\v\in \Hdiv(\Delta)$ and $q\in H^{1}(\Delta)$, then we denote
\begin{equation}
\int_{\partial \Delta} \big(\v\cdot \n\big) \, q\; dS \defining 
   \big\langle \v \cdot\n, q
      \big\rangle_{H^{-1/2}(\partial \Delta), H^{-1/2}(\partial \Delta)} .
\end{equation}
%
%In the sequel the following notation conventions will be understood
%
%\begin{subequations}
%
\item Since $\Gamma = \bigcup_{\sigma\,\in \,\E} \sigma$ we denote
\begin{equation}\label{Eq Integrals on the trace}
\int_{\Gamma} \left(\vone\cdot\n\right) \qtwo\, dS \defining 
\sum_{\sigma\, \in \, \edges} \int_{\sigma} \left(\vone\cdot\n\right) \qtwo\, dS .
\end{equation}
%
% \item For any open domain $\Delta$ 
\end{enumerate}
% %
\end{remark}
%
%
%%%%%%%%%%%%%%%%%%%%%%%%%%%%%%%%%%%%%%%%%%%%%%%%%%%%%%%%%%%%%%%%
\subsection{Weak Formulation of and Well-Posedness of the Problem}
%%%%%%%%%%%%%%%%%%%%%%%%%%%%%%%%%%%%%%%%%%%%%%%%%%%%%%%%%%%%%%%%
%
%
In this section we present a particular mixed-mixed formulation for the problem \eqref{Eq porous media strong decomposed}.
\begin{subequations}\label{Pblm weak continuous solution}
%
%\begin{equation*}
%\text{Find}\;(\u,\,p)\in \V\times Q
%\end{equation*}
%
\begin{multline}\label{Pblm weak continuous solution 1}
\text{Find}\;
%\bigg[\begin{array}{c}
%\uone \\
%\ptwo
%\end{array}\bigg] \in \X ,
%\bigg[\begin{array}{c}
%\utwo \\
%\pone
%\end{array}\bigg]
%\in \Y: \quad
%%
\big([\uone, \ptwo],[\utwo, \pone]\big)\in \X\times \Y : \quad
\int_{\Omega_1}  a \, \uone \cdot \vone 
%+ \int_{\Gamma}  \alpha  (\uone \cdot \n) (\vone\cdot \n) dS 
+ \int_{\Gamma}  \beta \, \ptwo \, \qtwo \, dS %\\
- \int_{\,\Gamma}\left(\uone\cdot\n\right) \qtwo\, d S
+\int_{\,\Gamma}\ptwo \left(\vone\cdot\n\right)\, d S \\
- \int_{\Omega_1} \pone \,\div\vone\,
- \int_{\Omega_{2}} \utwo \cdot \grad \qtwo\, %\\
= \int_{\Omega_{2}}F\, \qtwo - \int_{\Omega_{1}} \g \cdot \vone
+ \int_{\Gamma} \stress\, (\vone\cdot\n)\, dS 
- \int_{\Gamma} \flux\, \qtwo\, dS ,
\end{multline}
\begin{equation}\label{Pblm weak continuous solution 2}
\int_{\Omega_1}\div\uone\, \qone  %\,\\
+ \int_{\Omega_2} \grad \ptwo \cdot\vtwo 
+ \int_{\Omega_2}  a \, \utwo \cdot \vtwo\,
\\
= \int_{\Omega_{1}}F\, \qone 
- \int_{\Omega_{2}} \g \cdot \vtwo
\quad \text{ for all } 
\big([\vone, \qtwo],[\vtwo, \qone]\big)\in \X\times \Y .
%\bigg[\begin{array}{c}
%\vone \\
%\qtwo
%\end{array}\bigg] \in \X ,
%\bigg[\begin{array}{c}
%\vtwo \\
%\qone
%\end{array}\bigg]
%\in \Y .
\end{equation}
%
%\begin{equation*}
%\text{ for all }  (\v,\,q)\in \V\times Q
%\end{equation*}
%
\end{subequations}
Define the operators $\A: \X\rightarrow \X '$, $\B:\X\rightarrow \Y\,'$ and $\C:\Y\rightarrow \Y\,'$ by
\begin{subequations}\label{Def Operators Weak Continuous}
\begin{equation}\label{Def Regular Actions Operator }
\A [\vone, \qtwo],\big([\wone, \rtwo]\big)
%\bigg[\begin{array}{c}
%\vone \\
%\qtwo
%\end{array}\bigg]
%\left(\bigg[\begin{array}{c}
%\wone \\
%\rtwo
%\end{array}\bigg]\right)
\defining  
\int_{\Omega_1}  a \, \vone \cdot \wone 
%+  \int_{\Gamma} \alpha   (\vone \cdot \n) (\wone\cdot \n) dS 
+ \int_{\Gamma}  \beta \, \qtwo \, \rtwo \, dS \\
- \int_{\,\Gamma}\left(\vone\cdot\n\right) \rtwo\, d S
+\int_{\,\Gamma}\qtwo \left(\wone\cdot\n\right)\, d S ,
\end{equation}
\begin{equation}\label{Def Mixed Operator Continuous}
\B[\vone,\,\qtwo], \big([\wtwo,\,\rone]\big)\defining 
 \int_{\Omega_1}  \div\vone \, \rone
+ \int_{\Omega_{2}} \grad\qtwo \cdot \wtwo ,\,%\\
\end{equation}
\begin{equation}\label{Def Non Regular Actions non-negative Operator}
\C [\vtwo,\,\qone] \big([\wtwo,\,\rone]\big)\defining  \int_{\Omega_2} a \, \vtwo \cdot \wtwo .
\end{equation}
\end{subequations}
Thus, the Problem \eqref{Pblm weak continuous solution} is equivalent to
\begin{equation}\label{Pblm operators weak continuous solution}
%
%\begin{equation*}
%\text{Find a pair}\; (\u,\,p)\in \V\times Q\;\mathrm{sastisfying:}
%\end{equation*}
%
\begin{split}
%\begin{equation}\label{Pblm operators weak continuous solution 1}
\text{Find a pair}\; 
\big([\utwo,\,\pone], [\uone,\,\ptwo]\big)\in \X\times \Y: \quad 
\A[\utwo,\,\pone] + \B ' [\uone,\,\ptwo]  = F_{1}\quad \text{in}\; \X ' ,\\
%\end{equation}
%
%\begin{equation}\label{Pblm operators weak continuous solution 2}
- \B [\utwo,\,\pone]  + \C [\uone,\,\ptwo] = F_{2} \quad \text{in}\; \Y ' .
%\end{equation}
\end{split}
\end{equation}
Where $F_{1}\in \X '$ and $F_{2} \in \Y '$ are the functionals defined by the right hand side of \eqref{Pblm weak continuous solution 1} and \eqref{Pblm weak continuous solution 2} respectively. 
%
%
%\begin{proposition}\label{Th Neumann Data Estimated Solution}
%Let $\Delta$ be an open region, $\v\in \L^{2}(\Delta)$ and $\xi\in H_{0}^{1}(\Delta)$ be the unique solution to the problem
%%
%\begin{equation}\label{Pblm well posedness auxiliar problem}
%\begin{split}
%- \div\grad \xi  & = \div \v 
%\quad\text{in}\;\Delta\,,\\
% %
%\xi  = & 0 \quad\text{on} \;\partial \Delta .
%%\\
%%%%
%%\Vert \v\Vert_{\Hdiv( \Delta) } \leq 
%%C
%%& \left(
%%\Vert \rone\Vert_{0, \Delta}^{2} +
%%\Vert \rtwo\Vert_{0, \partial \Delta}^{2} \right)^{1/2}.
%\end{split}
%\end{equation}
%%
%Then, the projection of $\v$ onto $\V(\Delta)$ is given by $\grad \xi$.
%\end{proposition}
%%
%\begin{proof}
%%
%It is direct to see that 
%%
%%
%\qed
%\end{proof}
%
%
%
%
%%%%%%%%%%%%%%%%%%%%%%%%%%%%%%%%%%%%%%%%%%%%%%%%%%%%%%%%%%%%%%%%%%%%%%%%%%%%%%%%%%%%%%%%%%%%%%%
\subsubsection{Inf-Sup Condition of the Operator $\B$ and Coerciveness
of the Operator $\A$ on $\X\cap \ker (\B)$}
%%%%%%%%%%%%%%%%%%%%%%%%%%%%%%%%%%%%%%%%%%%%%%%%%%%%%%%%%%%%%%%%%%%%%%%%%%%%%%%%%%%%%%%%%%%%%%%
%
%
\begin{lemma}\label{Th inf sup condition}
The operator $\B: \X \rightarrow \Y '$ defined in equation \eqref{Def Mixed Operator Continuous} is continuous and satisfies the $\inf-\sup$ condition i.e., there exists a constant $C>0$ depending only on the map $\triang$ such that for every $[\wtwo, \rone]\in\Y $ there exists $[\vone, \qtwo]\in\X$ satisfying
\begin{equation}\label{Ineq existence of beta}
 \B\,[\vone, \qtwo] ([\wtwo, \rone])\geq 
C\, \big\Vert [\vone, \qtwo]\big\Vert_{\X} \big\Vert [\wtwo, \rone] \big\Vert_{\Y}
\, .
\end{equation}
Moreover, the constant $C>0$ is independent from $[\wtwo, \rone]$.
\end{lemma}
\begin{proof}
It is direct to see that the operator $\B$ is continuous. Now fix $[\wtwo, \rone]\;\in \Y$, for each polygon $L\in \tone$ let $\xi_{L} \in H_{0}^{1}(L)$ be the unique solution of the local homogeneous Dirichlet problem
\begin{align}
   & - \div \grad \xi_{L} = \rone \, \ind_{L}\quad \text{in}\; L ,& 
   & \xi_{L} = 0 \quad \text{on}\; \partial L .
\end{align}
Taking $\v_{L} \defining \grad \xi_{L}$ due to Poincar\'e inequality we observe that $\Vert  \v_{L} \Vert_{ \Hdiv(L) } \leq C_{L} \, \Vert  \rone \ind_{L} \Vert_{0, L}$ where $C_{L}$ depends only on the diameter of the simply connected region $L$. Therefore, the function $\vone \defining \sum_{L\, \in\, \tone} \v_{L}\, \ind_{L}$ clearly belongs to $\Hdiv(\Omega_{1})$ and $\Vert  \vone \Vert_{ \Hdiv(\Omega_{1}) } \leq \big(\max_{L\, \in \, \tone}  C_{L} \big)\, \Vert  \rone  \Vert_{0, \Omega_{1}}$. 

Let $\wtwo\in \V(\Omega_{2})$, by definition there must exist $\eta\in E(\Omega_{2})$ such that $\grad \eta = \vtwo$. Then, $\qtwo \defining \eta - \frac{1}{\vert \Omega_{2}\vert}\int_{\Omega_{2}} \eta$ belongs to the space $E_{0} (\Omega_{2})$ (defined in \eqref{Def Space Generating with Isomorphism}), it satisfies that $\grad \rtwo = \wtwo$ and due to the Inequality \eqref{Ineq Control on Space Generating with Isomorphism}, it holds that $\Vert \qtwo \Vert_{1, \Omega_{2}} \leq C \, \Vert \vtwo\Vert_{0, \Omega_{2}}$ with $C > 0$ depending only on the domain $\Omega_{2}$.

Hence, the pair $[\vone, \qtwo]$ belongs to $\X$ and $C\, \big\Vert [\vone, \qtwo] \big\Vert_{\X} \leq \big\Vert [\wtwo, \rone] \big\Vert_{\Y} $ for an adequate constant $C > 0$ independent from $[\wtwo, \rone]$. Moreover,
\begin{equation*}
   \B [\vone, \qtwo]\big([\wtwo, \rone]\big) = \big\Vert [\wtwo, \rone] \big\Vert_{\Y}^{2} 
   \geq C \, \big\Vert [\vone, \qtwo] \big\Vert_{\X} \, 
   \big\Vert [\wtwo, \rone] \big\Vert_{\Y} \,  .
\end{equation*}
The inequality above yields the inf-sup condition \eqref{Ineq existence of beta}. 
\qed
\end{proof}
\begin{proposition}\label{Th coercivity of A on ker B}
The operator $\A:\X\rightarrow \X\,'$ defined by \eqref{Def Regular Actions Operator } is $\X$-coercive on $\X\cap \ker (\B)$ i.e., 
\begin{equation}\label{Ineq charaterization of alpha}
\A[\vone, \qtwo] \big([\vone, \qtwo]\big)\geq  C
\big\Vert[\vone, \qtwo] \big\Vert_{\X}^{2} \,,\quad 
\text{for all}\; [\vone, \qtwo]\in \X\cap \ker (\B) .
\end{equation}
Where $C>0$ is an adequate constant depending only on the map $\triang$ and the storage coefficient $\beta$. 
\end{proposition}
\begin{proof} The continuity of the operator $\A$ follows applying the Cauchy-Schwartz inequality on each of its summands and noticing that the boundary terms involved can be controlled by the norm $\Vert \cdot \Vert_{\X}$. For the coerciveness of the operator, let $[\vone, \qtwo]\in \X\cap \ker(\B)$ then 
   \begin{equation}\label{Def testing the kernel of B}
      \B[\vone, \qtwo]\big([\wtwo, \rone]\big) = 0 \quad \text{for all}\; [\wtwo, \rone]\in \Y .
   \end{equation}
In particular, testing \eqref{Def testing the kernel of B} with $[\bm{0}, \rone]\in \Y$ we conclude that $\div \vone = 0$ since $\rone$ is an arbitrary element in $L^{2}(\Omega_{1})$. On the other hand, clearly $\grad \qtwo \in \V(\Omega_{2})$ and the pair $[\grad \qtwo, 0]\in \Y$ is eligible for testing \eqref{Def testing the kernel of B}. The test yields $\grad\qtwo = \bm{0}$ i.e., $\qtwo$ is constant inside $\Omega_{2}$. Hence 
\begin{equation*}
   \int_{\Gamma} \beta \, \qtwo^{2} = 
   \frac{\Vert \beta \, \ind_{\partial \Gamma}\Vert_{L^{1}(\Gamma)}}{\vert  \Omega_{2} \vert} \, 
   \Vert \qtwo \Vert_{0, \Omega_{2}}^{2} 
   = \frac{\Vert \beta \, \ind_{\Gamma}\Vert_{L^{1}(\Gamma)}}{\vert  \Omega_{2} \vert} \, 
   \Vert \qtwo \ind_{\Omega_{2}} \Vert_{1, \Omega_{2}}^{2} .
\end{equation*}
Using the previous observations we get that
\begin{equation*}
\begin{split}
\A[\vone, \qtwo] \big([\vone, \qtwo]\big) & = 
\int_{\Omega_1}  a\vone \cdot \vone 
+ \int_{\Gamma}  \beta \, \qtwo^{2}  \, dS \\
& \geq
\Big\Vert\frac{1}{a} \Big\Vert^{-1}_{L^{\infty}(\Omega)}\Vert \vone \Vert^{2}_{\Hdiv(\Omega_{1} ) } 
+ \frac{\Vert \beta \, \ind_{\Gamma}\Vert_{L^{1}(\Gamma)}}{\vert  \Omega_{2} \vert} \, 
   \Vert \qtwo  \Vert_{1, \Omega_{2} }^{2} \\
 &  \geq C \, \big\Vert [\vone, \qtwo]\big\Vert_{\X}^{2} \, ,
   \end{split}
\end{equation*}
where $C = \min \big\{\Vert\frac{1}{a}\Vert^{-1}_{L^{\infty}(\Omega)}, \, 
\vert  \Omega_{2} \vert^{-1}\, \Vert \beta \, \ind_{\Gamma}\Vert_{L^{1}(\Gamma)} \,  \big\}$. This completes the proof. 
\qed
\end{proof}
\begin{theorem}\label{Th well-posednes of the weak variational formulation}
The Problem \eqref{Pblm operators weak continuous solution} is well-posed and there exist a constant $C>0$ depending only on the map $\triang$ such that
\begin{equation*}
\big\Vert [\utwo,\,\pone]\big\Vert_{\X} 
+ \big\Vert[\uone,\,\ptwo]\big\Vert_{\Y} 
\leq C\, (\Vert F_{1}\Vert_{\X '} + \Vert F_{2}\Vert_{\Y '}).
\end{equation*}
\end{theorem}
\begin{proof}
The proof is a direct application of Theorem \ref{Th well posedeness mixed formulation classic} as all the required hypotheses are satisfied. 
\qed
\end{proof}
%
%
%
%
%%%%%%%%%%%%%%%%%%%%%%%%%%%%%%%%%%%%%%%%%%%%%%%%%%%%%%%%%%%%%%%%%%%%%%%%%%%%%%%%%%%%%%%%%%%%%%%
\subsection{Recovering the Strong Problem}
%%%%%%%%%%%%%%%%%%%%%%%%%%%%%%%%%%%%%%%%%%%%%%%%%%%%%%%%%%%%%%%%%%%%%%%%%%%%%%%%%%%%%%%%%%%%%%%
%
%
%
%
We begin this section with the strong problem that is modeled by the weak variational formulation \eqref{Pblm weak continuous solution}. The process shows mild restrictions on the forcing terms which are characterized below.
\begin{theorem}\label{Th weak formulation of continuos version}
The solution of the weak variational problem \eqref{Pblm weak continuous solution} is a strong solution of the problem \eqref{Eq porous media strong decomposed} with the forcing gravitation term $\g$ in the equation \eqref{Eq Darcy Strong decomposed 0} replaced by $P\g$, which denotes its orthogonal projection onto the space $\V(\Omega_{2})$. In particular if $\g\ind_{\Omega_{2}} \in \V(\Omega_{2})$ the weak solution is exactly the strong solution.
\end{theorem}
\begin{proof} First we focus on recovering the constitutive and conservative equations \eqref{Eq Darcy Strong decomposed 0}, \eqref{Eq conservative strong decomposed 0}, \eqref{Eq conservative strong decomposed 1} and \eqref{Eq Darcy Strong decomposed 1}, on the domains $\Omega_{1}$ and $\Omega_{2}$ respectively. 

   Choose $[\bm{0}, \qone ] \in \Y$ and test equation \eqref{Pblm weak continuous solution 2} to get $\int_{\Omega_{1}} \div \uone\, \qone = \int_{\Omega_{1}} F\, \qone$, for all $\qone \in L^{2}(\Omega_{1})$. This yields the strong conservation equation \eqref{Eq conservative strong decomposed 0}. Next choose $[\vtwo, 0] \in \Y$ and test equation \eqref{Pblm weak continuous solution 2} to get $\int_{\Omega_{2}} (\grad \ptwo + a \, \utwo)\cdot \vtwo = - \int_{\Omega_{2}} \g\cdot \vtwo$. Clearly $\grad \ptwo + a \, \utwo\in \V(\Omega_{2})$ and the previous equality holds for all $\vtwo\in \V(\Omega_{2})$. From here it follows that $\grad \ptwo + a \, \utwo = - P \g$, where $P\g$ is the orthogonal projection of $\g$ onto $\V(\Omega_{2})$. This gives the constitutive Darcy equation \eqref{Eq Darcy Strong decomposed 1} with the forcing term $\g$ replaced by its projection $P\g$. Now let $\Phi\in [C_{0}^{\infty}(\Omega_{1})]^{N}$ then, testing equation \eqref{Pblm weak continuous solution 1} with $[\Phi, 0]\in \X$ we get
   \begin{equation*}%\label{Pblm weak continuous solution 1}
\int_{\Omega_1}  a \, \uone \cdot \Phi 
- \int_{\Omega_1} \pone \,\div\Phi\,
= 
- \int_{\Omega_{1}} \g \cdot \Phi .
\end{equation*}
The above holds for all $\Phi\in [C_{0}^{\infty}(\Omega_{1})]^{N}$ then, the equation \eqref{Eq Darcy Strong decomposed 0} follows in the $\H^{-1}(\Omega_{1})$-sense. Moreover, recalling that $\uone, \g\in \L^{2}(\Omega_{1})$, the strong constitutive Darcy equation \eqref{Eq Darcy Strong decomposed 0} also holds in the $\L^{2}(\Omega_{1})$-sense. Now let $\varphi \in C_{0}^{\infty}(\Omega_{1})$ and test equation \eqref{Pblm weak continuous solution 1} with $[\bm{0}, \varphi]\in \X$ to get $- \int_{\Omega_{2}} \utwo \cdot \grad \varphi = \int_{\Omega_{2}}F\, \varphi$. Since this holds for all smooth functions the strong conservative equation \eqref{Eq conservative strong decomposed 1} follows. 

Next we focus on the boundary and the interface conditions. Take $\vone\in \Hdiv(\Omega_{1})$, test \eqref{Eq Darcy Strong decomposed 0} with $[\vone, 0]\in \X$, integrate by parts and get
\begin{equation*}%\label{Pblm weak continuous solution 1}
\int_{\Omega_1}  a\, \uone \cdot \vone 
+\int_{\,\Gamma}\ptwo \big(\vone\cdot\n \big)\, d S \\
+ \int_{\Omega_1} \grad \pone \cdot \vone
- \int_{\partial \Omega_1} \pone \,\big( \vone \cdot \outer \big) \, dS\,%\\
= - \int_{\Omega_{1}} \g \cdot \vone
+ \int_{\Gamma} \stress\, \big(\vone\cdot\n \big)\, dS .
\end{equation*}
We split the boundary term on $\partial \Omega_{1}$ in two pieces, the interfaces network $\Gamma$ and the outer part $\partial \Omega_{1} - \Gamma = \partial \Omega_{1} \cap \partial \Omega$ and replace $\outer$ with $\n$ using the relationship given in Definition \ref{Def geometric convention WVF}, equation \eqref{Def normal vector}.  Additionally, recalling that the identity \eqref{Eq Darcy Strong decomposed 0} is satisfied, the expression above writes as
\begin{equation*}%\label{Pblm weak continuous solution 1}
\int_{\,\Gamma}\ptwo \big(\vone\cdot\n \big)\, d S \\
- \int_{\Gamma} \pone \,\big(\vone \cdot \outer\big) \, dS
- \int_{\partial \Omega_1 \cap \partial \Omega } \pone \,\big(\vone \cdot \outer\big) \, dS\,%\\
= \int_{\Gamma} \stress\, \big(\vone\cdot\n \big)\, dS .
\end{equation*}
Since the above holds for all $\vone \in \Hdiv(\Omega_{1})$ and the map $\vone\mapsto \vone\cdot \n$ from $\Hdiv(\Omega_{1})$ onto $H^{-1/2}(\partial \Omega_{1})$ is surjective, the normal stress balance condition across the interfaces \eqref{Eq normal stress balance} and the drained Dirichlet boundary condition \eqref{Eq Drained Condition decomposed} follow in the sense of $H^{1/2}(\Gamma)$ and $H^{1/2}(\partial\Omega_{1}\cap \partial \Omega)$ respectively. Finally, taking $\qtwo\in H^{1}(\Omega_{2})$, testing \eqref{Eq Darcy Strong decomposed 0} with $[\bm{0}, \qtwo]\in \X$ and integrating by parts we get
\begin{equation}%\label{Pblm weak continuous solution 1}
\int_{\Gamma}  \beta \, \ptwo \, \qtwo \, dS %\\
- \int_{\,\Gamma}\big(\uone\cdot\n \big) \qtwo\, d S \\
+ \int_{\Omega_{2}} \div \utwo \, \qtwo
- \int_{\partial \Omega_{2}} \big(\utwo \cdot \outer\big) \qtwo\, dS%\\
= \int_{\Omega_{2}}F\, \qtwo 
- \int_{\Gamma} \flux\, \qtwo\, dS .
\end{equation}
Again, we split the boundary term on $\partial \Omega_{2}$ in two pieces, the interfaces network $\Gamma$ and the outer part $\partial\Omega_{2} - \Gamma = \partial \Omega_{2} \cap \partial \Omega$; then replace $\outer$ with $\n$ using the relationship given in Definition \ref{Def geometric convention WVF}, equation \eqref{Def normal vector}.  Since the identity \eqref{Eq conservative strong decomposed 1} is satisfied, the expression above reduces to
\begin{equation}%\label{Pblm weak continuous solution 1}
\int_{\Gamma}  \beta \, \ptwo \, \qtwo \, dS %\\
- \int_{\,\Gamma}\left(\uone\cdot\n\right) \qtwo\, d S \\
+ \int_{\Gamma} (\utwo \cdot \n) \qtwo\, dS
- \int_{\partial \Omega_{2} \cap \Omega} (\utwo \cdot \n) \qtwo\, dS 
%\\
= 
-\int_{\Gamma} \flux\, \qtwo\, dS .
\end{equation}
Observing that the above holds for all $\qtwo\in E(\Omega_{2})$, it follows that the normal flux balance condition across the interfaces \eqref{Eq normal flux balance} and the null normal flux boundary condition \eqref{Eq Non-Flux Condition decomposed} hold, in the sense of $H^{-1/2}(\Gamma)$ and $H^{-1/2}(\partial \Omega_{2}\cap \partial \Omega)$ respectively. This completes the proof.
   \qed
\end{proof}
Finally in order to identify which forcing terms can be modeled using this formulation, this section closes characterizing the orthogonal projection onto the spaces $\V (\Omega_{2} )$ and $\V^{\perp} (\Omega_{2} )$.
% %
% %
\begin{lemma}\label{Th orthogonal projections onto divergence free and orthogonal spaces}
	Let $\v\in \L^{2}(\Omega_{2})$. Let $\xi \in H_{0}^{1}(\Omega)\subseteq E(\Omega_{2})$ and $\eta\in E_{0}(\Omega_{2})$ be the unique solutions of the respective Dirichlet and Neumann problems
	\begin{subequations}\label{Pblm projection problems}
		\begin{align}\label{Pblm correction of the divergence}
		& - \div \grad \xi = - \div \v   \quad \text{in} \; \Omega_{2} , &
		& \xi = 0  \quad \text{on} \; \partial \Omega_{2} .
		\end{align}
		\begin{align}\label{Pblm correction of the normal trace}
		& - \div \grad \eta = 0   \quad \text{in} \; \Omega_{2} , &
		& \grad \eta\cdot \n = (\v - \grad \xi)\cdot\n  \quad \text{on} \; \partial \Omega_{2} .
		\end{align}
	\end{subequations}
	Then, $\v-\grad \xi - \grad \eta$ is the projection of $\v$ onto $\V^{\perp}(\Omega_{2})$ and $\grad \xi  + \grad \eta$ is the projection of $\v$ onto $\V(\Omega_{2})$.
\end{lemma}
\begin{proof} First recall that since $C_{0}^{\infty}(\Omega_{2})\subseteq E(\Omega_{2})$ then
\begin{equation*}
\V^{\perp}(\Omega) = 
\big\{\v\in \L^{2}(\Omega_{2}): \div \v = 0,\, \text{in }\, \Omega_{2}\; \text{and }
\, \v\cdot \n = 0\,, \; \text{on }\, \partial \Omega_{2}\big\},
\end{equation*}
i.e. $\V^{\perp}(\Omega_{2}) \subseteq \Hdiv(\Omega_{2})$. 
Next, observe that since $\xi\in H_{0}^{1}(\Omega_{2})$ then $\grad \xi \in \L^{2}(\Omega_{2})$, and because it is a solution to problem \eqref{Pblm correction of the divergence} it follows that $\div(\v - \grad \xi) = 0$ i.e., $\v - \grad \xi \in \Hdiv(\Omega_{2})$. Moreover, for any $q\in H^{1}(\Omega_{2})$ it holds that
	\begin{equation*}     
	\big\langle (\v -\grad \xi)\cdot\n, q
	\big\rangle_{H^{-1/2}(\partial \Omega_{2}), H^{-1/2}(\partial \Omega_{2})} = 
	\int_{\Omega_{2}} (\v -\grad \xi)\cdot \grad q  .
	\end{equation*}
	In particular, if $q = 1$ then $\big\langle (\v -\grad \xi)\cdot\n, 1\big\rangle_{H^{-1/2}(\partial \Omega_{2}), H^{-1/2}(\partial \Omega_{2})} = 0$ i.e., the data for the Neumann problem \eqref{Pblm correction of the normal trace} satisfy the compatibility condition and the problem has a unique solution $\eta$ in $E_{0}(\Omega_{2})$. Now, it is clear that $\v - \grad \xi -\grad \eta \in \H(\Omega_{2})$ and since $\v - (\v - \grad \xi -\grad \eta) = \grad \xi +\grad \eta$ is orthogonal to $\H(\Omega_{2})$, the result follows due to the characterization of orthogonal projections in Hilbert spaces. 
	\qed
\end{proof}
%
% 
%%
%\begin{remark}\label{Rem orthogonal projection of gravity}
%Notice that since $\L^{2}(\Omega_{2}) =  \bigoplus_{M \,\in \, \ttwo} \L^{2}(M)$ and $\V(\Omega_{2}) =  \bigoplus_{M \,\in \, \ttwo} \V(M)$, in particular, given $\g\in \L_{2}(\Omega_{2})$ we have that 
%%
%\begin{equation*}
%P\g = \sum_{M \,\in \, \ttwo} (\grad \xi_{M} + \grad \eta_{M}) \ind_{M} .
%\end{equation*}
%%
%Where, for each component $M$ of $\Omega_{2}$ the functions $\xi_{M}\in H_{0}^{1}(M)$, $\eta_{M}\in E(M)$ are those satisfying that $\grad \xi_{M} + \grad \eta_{M}$ is the projection of $\g\ind_{M}$ onto $\V(M)$ furnished by Lemma \ref{Th orthogonal projections onto divergence free and orthogonal spaces}.
%\end{remark}
%%
%%
%
%
%
%
% % % % % % % % % % % % % % % % % % % % % % % % % % % % % % % % % % % % % % % % % % % % %
\section{Conclusions and Final Discussion}\label{Sec Conclusion}
% % % % % % % % % % % % % % % % % % % % % % % % % % % % % % % % % % % % % % % % % % % % %
%
%
%The present work yields several accomplishments and also limitations as we point out in the following.
%
The present work yields several as summarized below.
\begin{enumerate}[(i)]
\item The study of communication in complex networks, as stated in \cite{EstradaComplexNetworks,ComplexNetworksMarteenVanSteen}, can be perplexing due to the large number of nodes and links. One of the greatest achievements of this article is to provide a relatively simple upscaled description to an otherwise very complicated study. This result gives the grounds to use well-studied numerical methods such as Finite Elements on Complex Network Theory.
	
\item The function $a(\cdot)$ presented in the strong form \eqref{Eq Darcy Strong decomposed 0}, \eqref{Eq Darcy Strong decomposed 0}, can be used as a scaling tool. If the volume of a region does not correctly reflect the impact of the universe it represents within the multiverse then, $a(\cdot)$ can be scaled in that region to counterbalance this deficiency and weigh each region of the map accordingly. 
%\item Consider the function $a$ presented in the strong form to be constant within each region $M$ and $L$ then it can act as a sort of friction coefficient. This has certain advantages, for instance, if one considers that the volume of a region should be greater then what a map provides (it may have more impact than what the map reflects) then instead of redrawing the map one may choose a greater value for $a$ to ensure that the variational formulation weighs each piece accordingly.

\item The interface conditions \eqref{Eq normal flux balance} and \eqref{Eq normal stress balance}, allow discontinuities along paths in the domain. This permits the modeling of sinks and/or sources on different manifolds along the region of interest.

\item The method for PDE analysis on graphs presented in this article and the previous achievements in the literature are dual concepts. The preexisting results rely on complex definitions for the operators of the PDE and simple definitions for the domain. This work introduces simple definitions for the operators of the PDE and very restrictive definitions for the graph. As many results in mathematics have shown, complementing dual concepts and a wise interplay between them, are the key for a very strong theory. The result presented proves that this method for analyzing PDE on graphs is the gateway for significantly deeper understanding in the field.
%\item The method we suggest for PDE analysis on graphs relies heavily on the coloring of the graph. This has a profound disadvantage with respect to more classical approaches for if one chooses to solve a particular PDE the variational formulation will not be universal, meaning, there will be a different formulation for a 3-coloreable graph than for a 4-coloreable graph. Classical methods do not rely on the geometry of the graph in order to provide a solution.

\item In the articles \cite{MoralesShow2} and \cite{Morales2}, a similar mixed variational formulation is introduced in order to model the saturated fluid flow within a fractured porous medium. This result is an alternative version, allowing to treat fractures with substantially more general geometry. Additionally, the mixed variational formulation presented here is, to the authors' best knowledge, unprecedented in the specialized literature of PDE analysis.  

\end{enumerate}
For the limitations of the method we point out the following ones.
\begin{enumerate}[(i)]
	\item There are two particularly important cases of discretization for 3-D polygonal domains in a bipartite fashion, using tetrahedra as discussed in \cite{MoralesOsorio}, or using in cubes. For both cases, it is not difficult to prove that a subdivision of the graph $K_{3,3}$ is contained in most of the graphs, defined by the ``tetrahedral grids" or the ``cubic grids", associated to a given domain. Therefore, most of the time, the associated graph is nonplanar and bipartite. Although in these scenarios the domain for the PDE setting is already defined, the nonplanar structure of their natural associated graphs suggests that the analysis for bipartite, nonplanar graphs is an important issue still to be addressed.      
	
	\item The variational formulation is dependent upon the domain chosen for the graph at hand. Thus, the construction of the domain that suits the application best will depend on the context. In this article we have suggested two possible choices, namely the tubular and downscaling maps.

\end{enumerate}
Finally, for further work we highlight the following aspects.
\begin{enumerate}[(i)]
	
	%\item Other than the tubular domain, we have not addressed the problem of associating a viable domain for non-planar graphs.
	%\item Other than the tubular domain, we have not addressed the problem of associating a viable domain for nonplanar graphs. Hence, the problem of  3-D simply connected domains partitioned in 3-D simply connected components for nonplanar graphs remains an open problem.
 	\item The method of analyzing highly clustered complex networks presented here is based upon the assumption that, approximating the behavior of a large number of particles with an open set, introduces an acceptable error when they are highly clustered. In order to understand the extent of this claim, further work about the error introduced by this approximation is of central importance. Notice that the quality of the approximation may very well depend on the nature of the quantities one is considering, which will impact in the PDE model. This aspect will be explored in future work.
	
	\item The Darcy Flow Model is a more general case of the Poisson equation whose unique differential operator is the Laplacian. Discrete versions of the Laplacian Operator for graphs are presented in \cite{GraphsPDE, WaveGraphs, AlgebraicGraph}. A natural conjecture is that the discrete Laplacian operator will approximate the continuous version if the number of nodes is very large. This conjecture will be addressed in future work, because it may provide the foundation to justify the claim that, jumping from discrete to continuous is a reasonable estimate for the global behavior when the number of nodes is large.
	
	\item The formulation introduced divides the domain in two types of regions, namely $\Omega_{1}$ and $\Omega_{2}$ which is the reason why the main setting of the problem had to be bipartite. Since every plane map is four coloreable it remains for future work  to find a formulation addressing up to four types of regions.
	%\item We were able to prove the well-possedness of the Darcy flow equation with two interface balance conditions, this created a natural division of the domain into two regions which is the reason why the setting for the main problem had to be bipartite. It is well known that every planar graph is 4-coloreable, and thus it remains as future work to find a well-possed variational formulation that would allow us to address four types of interface conditions.
	\end{enumerate}
\section*{Acknowledgements}
The authors wish to acknowledge Universidad Nacional de Colombia, Sede Medell\'in for its support in this work through the project HERMES 27798. They also wish to thank Tom\'as Mej\'ia for his constructive criticism and insightful suggestions in this work. 
\section*{References}
%
%
%\bibliography{bibliographie}
%\end{document}
%
%

%
%
\end{document}